\documentclass[11pt,a4paper]{article}
\usepackage[a4paper]{geometry}
\usepackage{amssymb,latexsym,amsmath,amsfonts,amsthm}
\usepackage{graphics}
\usepackage{epstopdf}
\usepackage{tikz}
\usepackage{url}
\usepackage{pgflibraryshapes}
\usetikzlibrary{arrows,decorations.markings}
\usepackage{subfigure}
\usepackage{comment,verbatim}
\usepackage{hyperref}

\DeclareMathOperator{\diag}{diag}

\renewcommand{\Re}{\mathrm{Re}\,}

\newcommand{\ud}{\,\mathrm{d}}

\def\supp{\mathop{\mathrm{supp}}\nolimits}

\newtheorem{theorem}{Theorem}[section]

\newtheorem{proposition}[theorem]{Proposition}

\newtheorem{rhp}[theorem]{RH problem}

\theoremstyle{definition}
\newtheorem{definition}[theorem]{Definition}

\theoremstyle{remark}

\numberwithin{equation}{section}

\begin{document}

\title{Mixed type multiple orthogonal polynomials associated with the modified Bessel functions and products of two
coupled random matrices}
\author{Lun Zhang\footnotemark[1]}
\date{\today}

\maketitle
\renewcommand{\thefootnote}{\fnsymbol{footnote}}
\footnotetext[1]{School of Mathematical Sciences and Shanghai Key Laboratory for Contemporary Applied Mathematics, Fudan University, Shanghai 200433, People's Republic of China. E-mail:
lunzhang@fudan.edu.cn}

\begin{abstract}
We consider mixed type multiple orthogonal polynomials associated with a system of weight functions consisting of two vectors. One vector is defined in terms of scaled modified Bessel function of the first kind $I_\mu$ and $I_{\mu+1}$, the other vector is defined in terms of scaled modified Bessel function of the second kind $K_\nu$ and $K_{\nu+1}$. We show that the corresponding mixed type multiple orthogonal polynomials exist. For the special case that each multi-index is on or close to the diagonal, basic properties of the polynomials and their linear forms are investigated, which include explicit formulas, integral representations, differential properties, limiting forms and recurrence relations. It comes out that, for specified parameters, the linear forms of these mixed type multiple orthogonal polynomials can be interpreted as biorthogonal functions encountering in recent studies of products of two coupled random matrices. This particularly implies a Riemann-Hilbert characterization of the correlation kernel, which provides an alternative way for further asymptotic analysis.

\vspace{2mm} \textbf{Keywords:} mixed type multiple orthogonal polynomials, modified Bessel functions,
integral representations, limiting forms, recurrence relations, random matrices
\end{abstract}

\section{Introduction}\label{sec:intro}
Multiple orthogonal polynomials are polynomials of one variable which are defined by orthogonality conditions with respect to several different weights. The general definition requires two sets of functions defined on the real axis $\mathbb{R}$:
\begin{equation*}
\textbf{w}_1=(w_{1,1},\ldots, w_{1,p}), \qquad \textbf{w}_2=(w_{2,1},\ldots, w_{2,q}),
\end{equation*}
where $p,q \in\mathbb{N}=\{1,2,3,\ldots\}$
and two multi-indices
$$\textbf{n}_1=(n_{1,1},\ldots, n_{1,p}) \in \mathbb{Z}_+^{p}, \qquad \textbf{n}_2=(n_{2,1},\ldots, n_{2,q}) \in \mathbb{Z}_+^{q},$$
where $\mathbb{Z}_+=\mathbb{N} \cup \{ 0 \}$. Following the usual vector notation, we set
$$|\textbf{n}_1|:=\sum_{i=1}^p n_{1,i},\qquad |\textbf{n}_2|:=\sum_{i=1}^q n_{2,i}.$$
It is assumed that
\begin{equation}\label{eq:multindexcond}
|\textbf{n}_1|=|\textbf{n}_2|+1,
\end{equation}
and each function $w_{1,i}w_{2,j}$, $i=1,\ldots,p$, $j=1,\ldots,q$ is a weight function on $\mathbb{R}$. Following \cite{DK,Sorokin94}, we have the following definitions.

\begin{definition}\label{def:MMOP}
Let $A_{\textbf{n}_1,\textbf{n}_2,1}, \ldots, A_{\textbf{n}_1,\textbf{n}_2,p}$ be a system of $p$ polynomials.
They are called mixed type multiple orthogonal polynomials with respect to the pair of multi-indices $(\textbf{n}_1,\textbf{n}_2)$ and the two vectors $\textbf{w}_1$ and $\textbf{w}_2$ if the following conditions hold:
\begin{enumerate}
  \item [$(\textrm i )$] $\deg A_{\textbf{n}_1,\textbf{n}_2,j}\leq n_{1,j}-1,$  \textrm{for $j=1,\ldots,p$, not all identically equal to zero (when $\deg A_{\textbf{n}_1,\textbf{n}_2,j}=-1$ we assume that $ A_{\textbf{n}_1,\textbf{n}_2,j} \equiv 0$)},
  \item [$(\textrm {ii} )$]
the linear form of these polynomials
\begin{equation}\label{def:Qn1}
Q_{\textbf{n}_1,\textbf{n}_2}(x):=\sum_{i=1}^p A_{\textbf{n}_1,\textbf{n}_2,i}(x)w_{1,i}(x)
\end{equation}
satisfies the orthogonality conditions
\begin{equation}\label{eq:orthocond}
\int Q_{\textbf{n}_1,\textbf{n}_2}(x)x^jw_{2,k}(x)\ud x=0, ~\textrm{for $k=1,\ldots,q$ and $j=0,1,\ldots,n_{2,k}-1$}.
\end{equation}
\end{enumerate}
\end{definition}

When $q=1$, the polynomials $A_{\textbf{n}_1,\textbf{n}_2,1}, \ldots, A_{\textbf{n}_1,\textbf{n}_2,p}$ are called type I multiple orthogonal polynomials, if $p=1$, the polynomial $A_{\textbf{n}_1,\textbf{n}_2,1}$ is called a type II multiple orthogonal polynomial. The standard orthogonality corresponds to the case $p=q=1$. Hence, the function $Q_{\textbf{n}_1,\textbf{n}_2}$ defined in \eqref{def:Qn1} admits a type I linear combination (with respect to $\textbf{n}_1$ and $\textbf{w}_1$), but satisfies type II orthogonality conditions (with respect to $\textbf{n}_2$ and $\textbf{w}_2$).

The equations \eqref{eq:orthocond} are a linear homogeneous system of $|\textbf{n}_2|$ equations for the $|\textbf{n}_1|$ unknown coefficients of the polynomials $A_{\textbf{n}_1,\textbf{n}_2,1},\ldots, A_{\textbf{n}_1,\textbf{n}_2,p}$. Due to the assumption \eqref{eq:multindexcond}, there is always a non-zero solution.

\begin{definition}
A pair of multi-indices $(\textbf{n}_1,\textbf{n}_2)$ is called \textit{normal}, if every solution to conditions $(\textrm{i})-(\textrm{ii})$ satisfies $\deg A_{\textbf{n}_1,\textbf{n}_2,j}=n_{1,j}-1,~j=1,\ldots,p$. If every pair of multi-indices $(\textbf{n}_1,\textbf{n}_2) \in \mathbb{Z}_+^{p} \times \mathbb{Z}_+^{q}$ is normal, we say that the pair $(\textbf{w}_1,\textbf{w}_2)$ is \textit{perfect}.
\end{definition}

For a normal pair of indices, it is easy to verify that the polynomials $A_{\textbf{n}_1,\textbf{n}_2,j}$, $j=1,\ldots,p$ are unique up to a constant factor.

From Definition \ref{def:MMOP}, it is clear that the roles of two vectors $\textbf{w}_1$ and $\textbf{w}_2$ are not symmetric. One could also swap their roles, which leads to the following notion of duality \cite{DS94}.

\begin{definition}\label{def:dualMOP}
Let $$\textbf{e}_j=(0,\ldots,0,1,0,\ldots,0)$$ be the $j$-th standard unit vector with $1$ on the $j$-th entry. Given a pair of multi-indices $(\textbf{n}_1,\textbf{n}_2)$ satisfying \eqref{eq:multindexcond}, the polynomials $B_{\textbf{n}_2+\textbf{e}_i,\textbf{n}_1-\textbf{e}_j,1}, \ldots, B_{\textbf{n}_2+\textbf{e}_i,\textbf{n}_1-\textbf{e}_j,q}$ are dual to the polynomials $A_{\textbf{n}_1,\textbf{n}_2,1}, \ldots, A_{\textbf{n}_1,\textbf{n}_2,p}$ if they are mixed type multiple orthogonal polynomials with respect to the pair of multi-indices $(\textbf{n}_2+\textbf{e}_i,\textbf{n}_1-\textbf{e}_j)$ and the two vectors $\textbf{w}_2$ and $\textbf{w}_1$ for some $\textbf{e}_i \in \mathbb{Z}_+^q$ and $\textbf{e}_j \in \mathbb{Z}_+^p$.
\end{definition}

Hence, if we define the following function similar to \eqref{def:Qn1}
\begin{equation}\label{def:Pn1}
P_{\textbf{n}_2+\textbf{e}_i,\textbf{n}_1-\textbf{e}_j}(x):=\sum_{k=1}^q B_{\textbf{n}_2+\textbf{e}_i,\textbf{n}_1-\textbf{e}_j,k}(x)w_{2,k}(x),
\end{equation}
it satisfies the orthogonality conditions
\begin{equation}
\int P_{\textbf{n}_2+\textbf{e}_i,\textbf{n}_1-\textbf{e}_j}(x)x^l w_{1,k}(x) \ud x =0
\end{equation}
for $k=1,\ldots,p$ and
$$l=\left\{
      \begin{array}{ll}
        0,1,\ldots,n_{1,k}-1, & \hbox{if $ k \neq j$,} \\
        0,1,\ldots,n_{1,j}-2, & \hbox{if $ k = j$.}
      \end{array}
    \right.
$$
This, together with \eqref{def:Qn1}, also implies that
\begin{equation}
\int  Q_{\textbf{r},\textbf{s}}(x)P_{\textbf{n}_2+\textbf{e}_i,\textbf{n}_1-\textbf{e}_j}(x) \ud x=0, ~\textrm{if $\textbf{r} \leq \textbf{n}_1-\textbf{e}_j$ or $\textbf{s} \geq \textbf{n}_2+\textbf{e}_i$},
\end{equation}
where the inequalities for multi-indices are understood in a componentwise manner.

If the pair of weights $(\textbf{w}_1,\textbf{w}_2)$ is perfect, so is the pair $(\textbf{w}_2,\textbf{w}_1)$; see \cite[Proposition 2.1]{FMM13}. Hence, for a normal pair of indices, the dual polynomials are also unique up to a multiplicative constant.

As a generalization of orthogonal polynomials, multiple orthogonal polynomials originated from Hermite-Pad\'{e} approximation in the context of irrationality and transcendence proofs in number theory. They are further developed in
approximation theory (cf. \cite{Apt98,WVA06} and references therein) and have played an important role nowadays in the studies of stochastic models (e.g. random matrix theory, non-intersecting paths, etc.) arising from mathematical physics;
cf. \cite{Kui10a,Kui10b} and references therein.

There are various families of multiple orthogonal polynomials which extend the classical orthogonal polynomials (cf. \cite{ABVAN03} and \cite[Chapter 23]{Ismail}). The aim of this paper, however, is to investigate mixed type multiple orthogonal polynomials associated with non-classical weights, namely, the modified Bessel functions. Our motivation is twofold:

On the one hand, although the polynomials orthogonal with respect to the modified Bessel functions in the usual sense do not exhibit nice properties as the classical ones, the associated multiple orthogonal polynomials indeed have some interesting properties, and remarkably, they are all closely related to some stochastic models. The type I and type II multiple orthogonal polynomials associated with the modified Bessel functions of the second kind were introduced and studied in \cite{CD00,CCVA08,CVA2001,VAY00,ZP}. Their connections to the products of independent Ginibre matrices have been revealed recently in \cite{Kuijlaars-Zhang14,Zhang13}. For the the modified Bessel functions of the first kind, the associated type I and type II multiple orthogonal polynomials were introduced and studied in \cite{CD00,CouVan03,CouVan03b,KuiRom10}. These polynomials are of significant assistance in establishing local universality in non-intersecting Bessel paths \cite{KMW09}; see also \cite{DKRZ12} for the case of mixed type. The difference between the present paper and those aforementioned is that the system of weight functions considered here is also of `mixed' type. One vector $(\omega_{\mu,a},\omega_{\mu+1,a})$ is defined in terms of scaled modified Bessel function of the first kind $I_\mu$ and $I_{\mu+1}$, while the other vector $(\rho_{\nu,b},\rho_{\nu+1,b})$ is defined in terms of scaled modified Bessel function of the second kind $K_\nu$ and $K_{\nu+1}$.

On the other hand, these mixed type multiple orthogonal polynomials arise naturally in recent studies of products of two coupled random matrices \cite{AS15,Liu16}. More precisely, let us consider, in a simple case, two independent matrices $A$ and $B$ of size $n\times M$ ($M\geq n$) with i.i.d. standard complex Gaussian entries. We then define two random matrices
\begin{equation}\label{eq:X1andX2}
X_1:=\frac{1}{\sqrt{2}}(A-i\sqrt{\tau}B),\qquad X_2:=\frac{1}{\sqrt{2}}(A^*-i\sqrt{\tau}B^*), \qquad 0< \tau < 1,
\end{equation}
where the superscript $^*$ stands for the conjugate transpose. Clearly, the matrices $X_1$ and $X_2$ are not independent anymore. It was shown in \cite{AS15} that the squared singular values of $X_1X_2$ form a determinantal point process with correlation kernel
\begin{equation}\label{def:Kn}
    K_n(x,y) = \sum_{k=0}^{n-1} \mathcal{Q}_k(x) \mathcal{P}_k(y),
    \end{equation}
where for each $k = 0, 1, \ldots$, $\mathcal{Q}_k$ belongs to the linear span of certain functions built from scaled modified Bessel functions of the first kind, while $\mathcal{P}_k$ belongs to the linear span of certain functions built from scaled modified Bessel functions of the second kind in such a way that
\begin{equation}\label{eq:biortho}
    \int_0^{\infty} \mathcal{Q}_k(x) \mathcal{P}_j(x) \ud x = \delta_{j,k}.
    \end{equation}
It comes out that the biorthogonal functions $\mathcal{Q}_k$ and $\mathcal{P}_k$ can be interpreted as the linear forms of our mixed type multiple orthogonal polynomials with specified parameters. In this case, some properties of $\mathcal{Q}_k$ and $\mathcal{P}_k$ have already been established in \cite{AS15}, but without noting the multiple orthogonality and the mixed type orthogonal polynomials. This connection also holds if the pair $(X_1,X_2)$ is distributed according to a coupled two-matrix matrix model introduced in \cite{Liu16}. We will come back to this issue at the end of this paper.

The rest of this paper is organized as follows. The precise definitions of the modified Bessel functions are given in Section \ref{sec:MBF}, where we also collect some of their basic properties for later use.
After fixing some notations used throughout this paper in Section \ref{sec:MMOP}, we show in Section \ref{sec:normality}
that the corresponding mixed type multiple orthogonal polynomials are uniquely determined except for a constant factor.
This mainly follows from the facts that both vectors form algebraic Chebyshev (AT) systems (actually, Nikishin systems),
which are already known for special cases \cite{CouVan03,VAY00}, and can be easily extended to general situations.
We include a proof based on the criteria established in \cite{CouVan03}. We then turn to the special case that each multi-index
is on or close to the diagonal. Basic properties of the polynomials as well as the dual polynomials, and their linear forms are
investigated, which include explicit formulas, integral representations,
differential properties (Section \ref{sec:expli formula}), limiting forms (Section \ref{sec:limitingform}) and recurrence
relations (Section \ref{sec:recurrence rel}). We finally explain how our mixed type multiple orthogonal polynomials are related
to products of two coupled random matrices in a general setting \cite{Liu16}. This connection particularly implies a Riemann-Hilbert (RH)
characterization of the correlation kernel, which provides an alternative way for further asymptotic analysis.

\section{Modified Bessel functions}\label{sec:MBF}

\subsection{Modified Bessel functions of the first kind}
The modified Bessel function of the first kind $I_\mu$ (see \cite[Section 10.25]{DLMF}) is defined by
\begin{equation}\label{def:I}
I_{\mu}(z)=\left(\frac{z}{2}\right)^\mu\sum_{k=0}^\infty\frac{(z^2/4)^k}{k!\Gamma(\mu+k+1)}, \quad \mu\in\mathbb{R},
\end{equation}
which is analytic in the complex plane with a cut along the negative real axis, and is a real positive function for $\mu>-1$ and $z>0$. The function $I_\mu$ satisfies the modified Bessel's equation
\begin{equation}\label{def:MBequation}
z^2\frac{\ud^2 u}{\ud z^2}+z\frac{\ud u}{\ud z}-(z^2+\mu^2)u=0.
\end{equation}

Some basic properties of $I_\mu$ that are relevant to this paper include:
\begin{itemize}
  \item recurrence relations (see \cite[Equation 10.29.1]{DLMF}):
\begin{align}
I_{\mu-1}(z)-I_{\mu+1}(z) & =\frac{2\mu}{z}I_{\mu}(z), \label{eq:recurI1}\\
I_{\mu-1}(z)+I_{\mu+1}(z) & = 2I_{\mu}'(z) \label{eq:recurI2};
\end{align}
  \item asymptotic behavior (see \cite[Equations 10.30.1 and 10.40.1]{DLMF}):
\begin{align}
I_{\mu}(z)&\sim\left(\frac{z}{2}\right)^\mu/\Gamma(\mu+1),\quad z\to 0, \quad \mu>-1,
\\
I_{\mu}(z)&= \frac{e^z}{\sqrt{2 \pi z}}\left(1+O\left(\frac{1}{z}\right)\right),\quad z\to \infty,
\quad |\arg z|<\frac{\pi}{2}. \label{eq:asyI2}
\end{align}
\end{itemize}

We now define scaled modified Bessel functions of the first kind $\omega_{\mu,a}$ $(\mu>-1,a>0)$ as follows:
\begin{equation}\label{def:omega}
\omega_{\mu,a}(x)=x^{\frac{\mu}{2}}I_{\mu}(2a\sqrt{x}), \quad x>0.
\end{equation}
Note that $\omega_{\mu,a}(x)$ is positive over the positive real axis. In view of \eqref{eq:recurI1}--\eqref{eq:asyI2}, it is readily seen that the functions $\omega_{\mu,a}$ satisfy
\begin{align}
\omega_{\mu+1,a}(x) & = x \omega_{\mu-1,a}(x)-\frac{\mu}{a}\omega_{\mu,a}(x),  \label{eq:recuromega1}\\
\omega_{\mu+1,a}'(x) & = a \omega_{\mu,a}(x), \label{eq:recuromega2}
\end{align}
and have the asymptotic behavior:
\begin{align}
\omega_{\mu,a}(x)&\sim (ax)^\mu/\Gamma(\mu+1),\quad x\to 0, \label{eq:asyomega1}
\\
\omega_{\mu,a}(x)&\sim \frac{x^{\frac{2\mu-1}{4}}e^{2a\sqrt{x}}}{2\sqrt{\pi a}},\quad x\to +\infty. \label{eq:asyomega2}
\end{align}

The recurrence relation \eqref{eq:recuromega1} particularly implies that $\omega_{\mu+m,a}$ can be expanded in terms of $\omega_{\mu,a}$ and $\omega_{\mu+1,a}$. Recall that the generalized hypergeometric function $_p F_q$ is defined by
 \begin{equation}\label{def:hypergeo}
 {\; }_p F_q \left({a_1,\ldots, a_p \atop b_1,\ldots,b_q} \Big{|} z \right)=\sum_{k=0}^\infty \frac{(a_1)_k\cdots (a_p)_k}{(b_1)_k \cdots (b_q)_k}\frac{z^k}{k!}
\end{equation}
with
\begin{equation}\label{eq:pochammer}
(a)_k=\frac{\Gamma(a+k)}{\Gamma(a)}=a(a+1)\cdots(a+k-1)
\end{equation}
being the Pochhammer symbol. We have 
\begin{proposition}\label{prop:omegahigh}
For any integer $m\geq 0$, it follows that
\begin{equation}\label{eq:recuromegam}
\omega_{\mu+m,a}(x)=a^{-m} r_{m,\mu}(a^2 x)\omega_{\mu,a}(x)+a^{1-m}s_{m,\mu}(a^2 x)\omega_{\mu+1,a}(x)
\end{equation}
where $ r_{0,\mu}(x)=1$, $ s_{0,\mu}(x)=0$, and for $m\geq 1$,
\begin{align}
r_{m,\mu}(x)&=-(-i)^mx^{m/2}h_{m-2,\mu+2}((2\sqrt{x}i)^{-1}) \nonumber \\
&=(-1)^m(\mu+2)_{m-2}{\; }_2 F_3 \left({-(m-2)/2,-(m-3)/2 \atop \mu+2,-m+2,1-m-\mu} \Big{|} 4x \right)x \nonumber \\
&=(-1)^m\sum_{j=0}^{\lfloor (m-2)/2 \rfloor}\binom{m-j-2}{j}(\mu+j+2)_{m-2j-2}x^{j+1}, \label{def:rmmu}
\end{align}
and
\begin{align}
s_{m,\mu}(x)&=(-i)^{m-1}x^{(m-1)/2}h_{m-1,\mu+1}((2\sqrt{x}i)^{-1}) \nonumber \\
&=(-1)^{m-1}(\mu+1)_{m-1}{\; }_2 F_3 \left({-(m-1)/2,-(m-2)/2 \atop \mu+1,-m+1,1-m-\mu} \Big{|} 4x \right) \nonumber \\
&=(-1)^{m+1}\sum_{j=0}^{\lfloor (m-1)/2 \rfloor}\binom{m-j-1}{j}(\mu+j+1)_{m-2j-1}x^{j}. \label{def:smmu}
\end{align}
Here, the functions $h_{m,\mu}$ are Lommel polynomials (see \cite[Sections 9.6--9.73]{Watson66}\footnotemark[2]) and $\lfloor x \rfloor=\max\{n\in\mathbb{Z}:n\leq x\}$ stands for the integer part of $x$.
\end{proposition}
\footnotetext[2] {In the notation $R_{m,\mu}$ of \cite{Watson66}, one has $h_{m,\mu}(x)=R_{m,\mu}(1/x)$.}
\begin{proof}
For the special case $a=1$, the relation \eqref{eq:recuromegam} is proved in \cite[Lemma 5]{CouVan03}, i.e.,
\begin{equation}\label{eq:reca=1}
\omega_{\mu+m,1}(x)= r_{m,\mu}(x)\omega_{\mu,1}(x)+s_{m,\mu}(x)\omega_{\mu+1,1}(x)
\end{equation}
with $ r_{0,\mu}(x)=1$, $ s_{0,\mu}(x)=0$, and for $m\geq 1$, $r_{m,\mu}(x)$ and $s_{m,\mu}(x)$
are given by \eqref{def:rmmu} and \eqref{def:smmu}, respectively.
This, together with the fact that
$$\omega_{\mu,1}(a^2x)=a^{\mu}\omega_{\mu,a}(x)$$
gives us \eqref{eq:recuromegam}.

This completes the proof of Proposition \ref{prop:omegahigh}.
\end{proof}
The Lommel polynomials $h_{m,\mu}$ satisfy the difference equation
$$h_{m+1,\mu}(x)=2x(m+\mu)h_{m,\mu}(x)-h_{m-1,\mu}(x)$$
with initial conditions $h_{-1,\mu}(x)=0$, $h_{0,\mu}(x)=1$, and have the hypergeometric representation (for $m>1$)
$$h_{m,\mu}(x)=(\mu)_m(2x)^m{\; }_2 F_3 \left({-m/2,-(m-1)/2 \atop \mu,-m,1-m-\mu} \Big{|} -1/x^2 \right);$$
see \cite[Chapter VI.6]{Chihara}.
The proof of \eqref{eq:reca=1} in \cite{CouVan03} relies on the facts that
\begin{equation}\label{eq:recuJ}
J_{\mu+m}(x)=h_{m,\mu}(1/x)J_{\mu}(x)-h_{m-1,\mu+1}(1/x)J_{\mu-1}(x)
\end{equation}
and
$$I_{\mu}(x)=e^{-(\mu \pi i)/2}J_{\mu}(xi),\quad x>0,$$
with $J_{\mu}$ being the Bessel function of the first kind (cf. \cite[Section 10.2]{DLMF}).

\subsection{Modified Bessel functions of the second kind}
The modified Bessel functions of the second kind $K_\nu$ (see \cite[Section 10.25]{DLMF}) are also known as Macdonald functions. $K_\nu$ and $I_\nu$ are two linearly independent solutions of \eqref{def:MBequation} (with $\mu$ replaced by $\nu$).  Like $I_\nu(z)$, $K_{\nu}(z)$ is also analytic in the complex plane with a cut along the negative real axis, and is a real positive function for $\nu>-1$ and $z>0$.

Some basic properties of $K_\nu$ that are relevant to this paper include:
\begin{itemize}
  \item integral representations of Mellin-Barnes type (see \cite[Equation 10.32.13]{DLMF}):
\begin{equation}\label{eq:KMellin}
K_\nu(z)=\frac{(z/2)^\nu}{4 \pi i}\int_{c-i\infty}^{c+i\infty}\Gamma(t)\Gamma(t-\nu)\left(\frac{z}{2}\right)^{-2t}\ud t, \end{equation}
for $c>\max\{\Re \nu, 0\}$ and $|\arg z|<\pi$;
  \item recurrence relations (see \cite[Equation 10.29.1]{DLMF}):
\begin{align}
K_{\nu-1}(z)-K_{\nu+1}(z) & =-\frac{2\nu}{z}K_{\nu}(z), \label{eq:recurK1}\\
K_{\nu-1}(z)+K_{\nu+1}(z) & = -2K_{\nu}'(z) \label{eq:recurK2};
\end{align}
  \item asymptotic behavior (see \cite[Equations 10.30.2, 10.30.3 and 10.40.2]{DLMF}):
\begin{align}
K_{\nu}(z)&\sim\left\{
                 \begin{array}{ll}
                   2^{\nu-1}\Gamma(\nu)z^{-\nu}, & \hbox{$z\to 0$, $~\Re\nu>0$,} \\
                   -\ln z, & \hbox{$z\to 0$, $~\nu=0$,}
                 \end{array}
               \right.
\\
K_{\nu}(z)& = \sqrt{\frac{\pi}{2z}}e^{-z}\left(1+O\left(\frac{1}{z}\right)\right),\quad z\to \infty, \quad |\arg z|<\frac{3\pi}{2}. \label{eq:asyK2}
\end{align}
\end{itemize}

We then define scaled modified Bessel functions of the second kind $\rho_{\nu,b}$ $(\nu>0,b>0)$ as follows:
\begin{equation}\label{def:rho}
\rho_{\nu,b}(x)=x^{\frac{\nu}{2}}K_{\nu}(2b\sqrt{x}), \quad x>0.
\end{equation}
The function $\rho_{\nu,b}(x)$ is positive over the positive real axis. By \eqref{eq:KMellin}--\eqref{eq:asyK2}, it follows that
\begin{equation}\label{eq:rhoMellin}
\rho_{\nu,b}(x)=\frac{b^{-\nu}}{4 \pi i}\int_{c-i\infty}^{c+i\infty}\Gamma(t)\Gamma(t+\nu)\left(b^2 x\right)^{-t}\ud t, \quad c>0.
\end{equation}
satisfy the recurrence relations
\begin{align}
\rho_{\nu+1,b}(x) & = x \rho_{\nu-1,b}(x)+\frac{\nu}{b}\rho_{\nu,b}(x),  \label{eq:recurrho1}\\
\rho_{\nu+1,b}'(x) & = -b \rho_{\nu,b}(x), \label{eq:recurrho2}
\end{align}
and have the asymptotic behavior:
\begin{align}
\rho_{\nu,b}(x)&\sim \frac{\Gamma(\nu)}{2b^\nu},\quad x\to 0, \label{eq:asyrho1}
\\
\rho_{\nu,b}(x)&\sim \left(\frac{\pi}{4b}\right)^{\frac{1}{2}}x^{\frac{2\nu-1}{4}}e^{-2b\sqrt{x}},\quad x\to +\infty. \label{eq:asyrho2}
\end{align}

A result similar to Proposition \ref{prop:omegahigh} holds for $\rho_{\nu,b}$.
\begin{proposition}\label{prop:rhohigh}
For any integer $m\geq 0$, we have
\begin{equation}\label{eq:rhohighrecu}
\rho_{\nu+m,b}(x)=(-b)^{-m} r_{m,\nu}(b^2 x)\rho_{\nu,b}(x)+(-b)^{1-m}s_{m,\nu}(b^2 x)\rho_{\nu+1,b}(x),
\end{equation}
where $ r_{0,\nu}(x)=1$, $ s_{0,\nu}(x)=0$, and for $m\geq 1$, the functions $r_{m,\nu}$ and $s_{m,\nu}$ are given by \eqref{def:rmmu} and \eqref{def:smmu}, respectively.
\end{proposition}
\begin{proof}
Since $\rho_{\nu,1}(b^2x)=b^{\nu}\rho_{\nu,b}(x)$, it is sufficient to prove \eqref{eq:rhohighrecu} for the case that $b=1$, which is similar to the proof of \eqref{eq:reca=1}.

If $b=1$, by \eqref{def:rho} and setting $y=2\sqrt{x}$, we could rewrite \eqref{eq:rhohighrecu} as
$$
K_{\nu+m}(y)=\left(-\frac{2}{y}\right)^{m} r_{m,\nu}\left(\frac{y^2}{4}\right)K_{\nu}(y)+
\left(-\frac{2}{y}\right)^{m-1}s_{m,\nu}\left(\frac{y^2}{4}\right)K_{\nu+1}(x).
$$
From \cite[Equation 10.27.8]{DLMF}, we note that
$$K_\nu(y)=\frac{\pi i}{2}e^{\nu \pi i/2}H_{\nu}^{(1)}(iy), \qquad x>0,$$
where $H_{\nu}^{(1)}(y)$ is the Bessel function of the third kind (see \cite[Section 10.2]{DLMF}). Hence, we further have
\begin{equation}\label{eq:recuH1}
H_{\nu+m}^{(1)}(iy)=\left(\frac{2i}{y}\right)^{m} r_{m,\nu}\left(\frac{y^2}{4}\right)H_{\nu}^{(1)}(iy)+
\left(\frac{2i}{y}\right)^{m-1}s_{m,\nu}\left(\frac{y^2}{4}\right)H_{\nu+1}^{(1)}(iy).
\end{equation}
Since $H_{\nu}^{(1)}$ satisfies the same three-term recurrence relation as $J_{\nu}$ (see \cite[Equation 10.6.1]{DLMF}),
it is then readily seen from \eqref{eq:recuJ} that \begin{equation}\label{eq:recuH}
H_{\nu+m}^{(1)}(iy)=-h_{m-2,\nu+2}(-i/y)H_{\nu}^{(1)}(iy)+h_{m-1,\nu+1}(-i/y)H_{\nu+1}^{(1)}(iy).
\end{equation}
Comparing \eqref{eq:recuH1} with \eqref{eq:recuH} gives us \eqref{eq:rhohighrecu} with $b=1$.

This completes the proof of Proposition \ref{prop:rhohigh}.
\end{proof}


\section{Mixed type multiple orthogonal polynomials associated with the modified Bessel functions}
\label{sec:MMOP}
With $\omega_{\mu,a}$ and $\rho_{\nu,b}$ defined in \eqref{def:omega} and \eqref{def:rho}, we set two vectors of functions
\begin{equation}\label{def:MBweights}
\textbf{w}_1=(\omega_{\mu,a},\omega_{\mu+1,a}), \qquad \textbf{w}_2=(\rho_{\nu,b},\rho_{\nu+1,b}).
\end{equation}
Given a pair of multi-indices $(\textbf{n}_1,\textbf{n}_2)\in\mathbb{Z}_+^2 \times \mathbb{Z}_+^2$ satisfying \eqref{eq:multindexcond}, we will study mixed type multiple orthogonal polynomials with respect to $(\textbf{n}_1,\textbf{n}_2)$ and the system of weights $(\textbf{w}_1,\textbf{w}_2)$ over $(0,\infty)$ under the conditions that
\begin{equation}\label{eq:paracond}
\mu>-1,\qquad \nu>0, \qquad b>a>0.
\end{equation}

By Definition \ref{def:MMOP}, we look for a vector of polynomials $(A_{\textbf{n}_1,\textbf{n}_2,1},  A_{\textbf{n}_1,\textbf{n}_2,2})$ such that
$$\deg A_{\textbf{n}_1,\textbf{n}_2,j}\leq n_{1,j}-1,~~~~j=1,2,$$
and satisfies the orthogonality conditions
\begin{multline}\label{eq:orthoQ1}
\int_0^\infty (A_{\textbf{n}_1,\textbf{n}_2,1}(x)\omega_{\mu,a}(x)+A_{\textbf{n}_1,\textbf{n}_2,2}(x)\omega_{\mu+1,a}(x))x^j\rho_{\nu,b}(x)\ud x \\
=\int_0^\infty Q_{\textbf{n}_1,\textbf{n}_2}(x)x^j\rho_{\nu,b}(x)\ud x
=0, ~~\textrm{$j=0,1,\ldots,n_{2,1}-1$}
\end{multline}
and
\begin{multline}\label{eq:orthoQ2}
\int_0^\infty (A_{\textbf{n}_1,\textbf{n}_2,1}(x)\omega_{\mu,a}(x)+A_{\textbf{n}_1,\textbf{n}_2,2}(x)\omega_{\mu+1,a}(x))x^j\rho_{\nu+1,b}(x)\ud x \\
=\int_0^\infty Q_{\textbf{n}_1,\textbf{n}_2}(x)x^j\rho_{\nu+1,b}(x)\ud x
=0, ~~\textrm{$j=0,1,\ldots,n_{2,2}-1$.}
\end{multline}
The vector of dual polynomials $(B_{\textbf{n}_1,\textbf{n}_2,1}, B_{\textbf{n}_1,\textbf{n}_2,2})$ and their linear form $P_{\textbf{n}_1,\textbf{n}_2}$ are then defined similarly according to Definition \ref{def:dualMOP} and \eqref{def:Pn1}. Note that the assumptions \eqref{eq:paracond} ensure that all the relevant integrals are well-defined; see the asymptotic behavior of $\omega_{\mu,a}$ and $\rho_{\nu,b}$ given in \eqref{eq:asyomega1}--\eqref{eq:asyomega2} and \eqref{eq:asyrho1}--\eqref{eq:asyrho2}, respectively.

In what follows, we first deal with uniqueness of the polynomials and show that they are indeed unique up to a multiplicative factor. Then, we will focus on the special case that each multi-index is on or close to the diagonal. For that purpose, the following notations will be used throughout this paper. For each $ m \in \mathbb{Z}_+$, we set
\begin{equation}
\textbf{m}=\left(\lfloor \frac{m}{2} \rfloor+1, \lfloor \frac{m-1}{2} \rfloor+1\right)
=\left\{
 \begin{array}{ll}
\left( \frac{m}{2}+1,                                                                                         \frac{m}{2}\right), & \hbox{ $m$ even,} \\
\left( \frac{m+1}{2},                                                                                         \frac{m+1}{2}\right), & \hbox{ $m$ odd,}                                                                                        \end{array}
\right.
\end{equation}
Hence, $|\textbf{m}|=m+1$. We then define, for each $n\in\mathbb{N}$,
\begin{align}
&A_{n,i}(x)=A_{\textbf{n},\textbf{n-1},i}(x)=\left\{
                                               \begin{array}{ll}
                                                 A_{\left( \frac{n+1}{2},                                                                                         \frac{n+1}{2}\right),\left( \frac{n+1}{2},                                                                                         \frac{n-1}{2}\right),i}(x), & \hbox{$n$ odd,} \\
                                                 A_{\left( \frac{n}{2}+1,                                                                                         \frac{n}{2}\right),\left( \frac{n}{2},                                                                                         \frac{n}{2}\right),i}(x), & \hbox{$n$ even,}
                                               \end{array}
                                             \right.
~~ i=1,2, \label{def:An}\\
& Q_n(x) \nonumber \\
&=A_{n,1}(x)\omega_{\mu,a}(x)+A_{n,2}(x)\omega_{\mu+1,a}(x) \label{def:QnNotation} \\
&=\left\{
                                               \begin{array}{ll}
                                                 A_{\left( \frac{n+1}{2},                                                                                         \frac{n+1}{2}\right),\left( \frac{n+1}{2},                                                                                         \frac{n-1}{2}\right),1}(x)\omega_{\mu,a}(x)+A_{\left( \frac{n+1}{2},                                                                                         \frac{n+1}{2}\right),\left( \frac{n+1}{2},                                                                                         \frac{n-1}{2}\right),2}(x)\omega_{\mu+1,a}(x), & \hbox{ $n$ odd,}  \\
                                                 A_{\left( \frac{n}{2}+1,                                                                                         \frac{n}{2}\right),\left( \frac{n}{2},                                                                                         \frac{n}{2}\right),1}(x)\omega_{\mu,a}(x)+A_{\left( \frac{n}{2}+1,                                                                                         \frac{n}{2}\right),\left( \frac{n}{2},                                                                                         \frac{n}{2}\right),2}(x)\omega_{\mu+1,a}(x), & \hbox{ $n$ even,}
                                               \end{array}
                                             \right.  \nonumber
\end{align}
and similarly,
\begin{align}
B_{n,i}(x)&=B_{\textbf{n},\textbf{n-1},i}(x), ~~ i=1,2, \label{def:Bn}\\
P_n(x)&=B_{n,1}(x)\rho_{\nu,b}(x)+B_{n,2}(x)\rho_{\nu+1,b}(x) \label{def:PnNotation}.
\end{align}
To emphasize the dependence of parameters $\mu,\nu,a,b$, we will occasionally write
\begin{equation}\label{eq:notations}
\begin{aligned}
A_{n,i}(x)&=A_{n,i}^{\mu,\nu,a,b}(x),\qquad Q_n(x)=Q_n^{\mu,\nu,a,b}(x),
\\
B_{n,i}(x)&=B_{n,i}^{\mu,\nu,a,b}(x), \qquad P_n(x)=P_n^{\mu,\nu,a,b}(x).
\end{aligned}
\end{equation}

\section{Normality}\label{sec:normality}

The main result of this section is the following theorem:
\begin{theorem}\label{thm:nomality}
Let $\textbf{w}_1$ and $\textbf{w}_2$ be two vectors of scaled modified Bessel functions given in \eqref{def:MBweights}. We have that any pair of multi-indices $(\textbf{n}_1,\textbf{n}_2)\in\mathbb{Z}_+^2 \times \mathbb{Z}_+^2$ is normal, and the vector of mixed type multiple orthogonal polynomials $(A_{\textbf{n}_1,\textbf{n}_2,1},  A_{\textbf{n}_1,\textbf{n}_2,2})$ is uniquely determined except for a constant factor. Furthermore, the function $Q_{\textbf{n}_1,\textbf{n}_2}$ has $|\textbf{n}_2|$ sign changes on $(0,\infty)$.
\end{theorem}

The essential issue in the proof of Theorem \ref{thm:nomality} is that both $\textbf{w}_1$ and $\textbf{w}_2$ form the so-called algebraic Chebyshev (AT) systems (cf. \cite[Section 23.1.2]{Ismail}, \cite{NikSor}), which we recall now.

\begin{definition}\label{def:AT}
A system of $m$ real-valued continuous functions $(w_1,\ldots,w_m)$ defined on the interval $\Delta\subset \mathbb{R}$  forms an AT system for the multi-index $\textbf{n}=(n_1,\ldots,n_m)\in\mathbb{Z}_+^m$ if for any choice of polynomials
$p_{1}(x),\ldots,p_m(x)$ with $\deg p_i\leq n_j-1$, the function
\[  \sum_{i=1}^m p_i(x)w_i(x) \]
has at most $|\textbf{n}|-1$ zeros on $\Delta$. If this is true for all $\textbf{n}\in\mathbb{Z}_+^m$, we have an AT system on $\Delta$.
\end{definition}

If a system of weight functions forms an AT system, the associated type I and type II multiple orthogonal polynomials exist uniquely, up to a multiplicative constant; cf. \cite[Section 23.1.2]{Ismail}.

For a system of functions consisted of two members $(w_1,w_2)$, the following theorem gives a sufficient condition that
$(w_1,w_2)$ forms an AT system, which is adapted from \cite[Theorems 2 and 3]{CouVan03}.
\begin{theorem}\label{thm:CouVan}
Let $w_1$ and $w_2$ be two functions defined on the real interval $\Delta_1$ with $w_1(x)>0$ and $w_2(x)>0$ for $x\in\Delta_1$. Suppose that there exist two measures $\sigma_1$ and $\sigma_2$ on $\Delta_2$, a real interval with $\overset{\circ}{\Delta}_1 \cap \overset{\circ}{\Delta}_2=\emptyset$, such that
\begin{align}
\frac{w_2(x)}{w_1(x)}&=\int_{\Delta_2}\frac{\ud \sigma_1(t)}{x-t}, \label{eq:Nikishin} \\
\frac{w_1(x)}{w_2(x)}&=x\int_{\Delta_2}\frac{\ud \sigma_2(t)}{x-t} \label{eq:Nikishin2}.
\end{align}
Then, $(w_1,w_2)$ forms an AT system on $\Delta_1$, provided both $\supp(\sigma_1)$ and $\supp(\sigma_2)$ contain an infinite number of points.
\end{theorem}

A system of functions $(w_1,w_2)$ defined on $\Delta_1$ such that the ratio $w_2/w_1$ can be written as a Markov function for a measure $\sigma_1$ on $\Delta_2$ \eqref{eq:Nikishin} with $\overset{\circ}{\Delta}_1 \cap \overset{\circ}{\Delta}_2=\emptyset$ is called a Nikishin system (with 2 functions), due to the first description by Nikishin \cite{Niki82}. The proof of Theorem \ref{thm:CouVan} in \cite{CouVan03} is direct and relies on determinantal identities for structured matrices. Indeed, the condition \eqref{eq:Nikishin} ensures that $(w_1,w_2)$ forms an AT system for every $(n,m)\in\mathbb{N}^2$ with $m\leq n+1$, while $\eqref{eq:Nikishin2}$ implies that $(w_1,w_2)$ forms an AT system for every $(n,m)\in\mathbb{N}^2$ with $ n \leq m$, which together yields the assertion. The perfectness of a Nikishin system of two functions was also proved in \cite{BLOP92,DK} based on Hermite-Pad\'{e} approximation.

We are now ready to prove Theorem \ref{thm:nomality}.

\paragraph{Proof of Theorem \ref{thm:nomality}}
We show that both $(\omega_{\mu,a},\omega_{\mu+1,a})$ and $(\rho_{\nu,b},\rho_{\nu+1,b})$ form AT systems, which are already known for the special cases $a=1$ and $b=1$ \cite{CouVan03,VAY00}. We follow the idea in \cite{CouVan03}.

For the pair $(\omega_{\mu,a},\omega_{\mu+1,a})$, let $\{j_{\nu,n}\}_{n=1}^\infty$ be the $n$-th positive zero of the Bessel function of the first kind $J_\nu$. It is shown in \cite[Theorem 4.7]{IsmailKel79} that
\begin{equation*}\label{eq:IsmKel}
z^{(\nu-\mu)/2}\frac{I_{\mu}(\sqrt{z})}{I_{\nu}(\sqrt{z})}=-2\sum_{n=1}^{\infty}\frac{j_{\nu,n}^{\nu+1-\mu}J_\mu(j_{\nu,n})}
{(z+j_{\nu,n}^2)J'_\nu(j_{\nu,n})},\quad \mu>\nu>-1,\quad |\arg z|<\pi.
\end{equation*}
This, together with the definition of $\omega_{\mu,a}$ given in \eqref{def:omega}, implies that
\begin{multline}\label{eq:rationomega1}
\frac{\omega_{\mu+1,a}(x)}{\omega_{\mu,a}(x)}=\sqrt{x}\frac{I_{\mu+1}(2a\sqrt{x})}{I_{\mu}(2a\sqrt{x})}
\\
=-x\sum_{n=1}^{\infty}\frac{J_{\mu+1}(j_{\mu,n})}
{(ax+j_{\mu,n}^2/(4a))J'_\mu(j_{\mu,n})}
=x\int_{-\infty}^{0}\frac{\ud \sigma_1(t)}{x-t},
\end{multline}
where
$$\sigma_1=-\sum_{n=1}^{\infty}\frac{J_{\mu+1}(j_{\mu,n})}
{aJ'_\mu(j_{\mu,n})}\delta_{-\frac{j_{\mu,n}^2}{4a^2}}$$
is a discrete measure on $(-\infty,0]$ built from a linear combination of Dirac measures on the points $-j_{\mu,n}^2/4a^2$. On the other hand, it is readily seen from \eqref{eq:recuromega1} and \eqref{eq:rationomega1} that
\begin{multline}\label{eq:rationomega2}
\frac{\omega_{\mu,a}(x)}{\omega_{\mu+1,a}(x)}=\frac{\mu+1}{ax}+\frac{1}{x}\frac{\omega_{\mu+2,a}(x)}{\omega_{\mu+1,a}(x)}
\\=\frac{\mu+1}{a}\frac{1}{x}-\sum_{n=1}^{\infty}\frac{J_{\mu+2}(j_{\mu+1,n})}
{(ax+j_{\mu+1,n}^2/(4a))J'_{\mu+1}(j_{\mu+1,n})}
= \int_{-\infty}^{0}\frac{\ud \sigma_2(t)}{x-t},
\end{multline}
where
$$\sigma_2=\frac{\mu+1}{a}\delta_0-\sum_{n=1}^{\infty}\frac{J_{\mu+2}(j_{\mu+1,n})}
{aJ'_{\mu+1}(j_{\mu+1,n})}\delta_{-\frac{j_{\mu+1,n}^2}{4a^2}}.$$

For the pair $(\rho_{\nu,a},\rho_{\nu+1,a})$, we make use of the following formula (see \cite{Grosswald} and \cite{Ismail77}):
$$ z^{-1/2}\frac{K_{\nu-1}(\sqrt{z})}{K_\nu(\sqrt{z})}=\frac{2}{\pi^2}\int_0^\infty \frac{1}{(z+t)(J_\nu^2(\sqrt{t})+Y_{\nu}^2(\sqrt{t}))t}\ud t,\quad \nu>0,\quad |\arg z|<\pi,
$$
where
$$Y_{\nu}(z)=\left\{
               \begin{array}{ll}
                 \frac{J_\nu(z)\cos(\nu \pi)-J_{-\nu}(z)}{\sin(\nu\pi)}, & \hbox{$\nu \neq 0,\pm1,\pm2,\ldots$, } \\
                 \left.\frac{1}{\pi}\frac{\partial J_\nu (z)}{\partial \nu}\right|_{\nu=n}+\left.\frac{(-1)^n}{\pi}\frac{\partial J_\nu (z)}{\partial \nu}\right|_{\nu=-n}, & \hbox{$\nu=n=0,\pm1,\pm2,\ldots$,}
               \end{array}
             \right.
$$
is the Bessel function of the second kind (Weber's Function); see \cite[Section 10.2]{DLMF}. Thus, by \eqref{def:rho}, it follows that
\begin{multline}\label{eq:ratiorho1}
\frac{\rho_{\nu,b}(x)}{\rho_{\nu+1,b}(x)}=\frac{K_{\nu}(2b\sqrt{x})}{\sqrt{x}K_{\nu+1}(2b\sqrt{x})}
=\frac{4b}{\pi^2}\int_0^\infty \frac{1}{(4b^2 x+t)(J_{\nu+1}^2(\sqrt{t})+Y_{\nu+1}^2(\sqrt{t}))t}\ud t
\\=\frac{1}{b \pi^2}\int_0^\infty \frac{1}{(x+t)(J_{\nu+1}^2(2b\sqrt{t})+Y_{\nu+1}^2(2b\sqrt{t}))t}\ud t =\int_{-\infty}^{0}\frac{\ud \sigma_3(t)}{x-t},
\end{multline}
where we have made use a change of variable $t=4b^2t$ in the third equality and
$$\ud \sigma_3 =-\frac{\ud t}{b \pi^2(J_{\nu+1}^2(2b\sqrt{-t})+Y_{\nu+1}^2(2b\sqrt{-t}))t}$$
is a continuous measure on $(-\infty,0]$. This, together with \eqref{eq:recurrho1}, also implies that
\begin{equation}\label{eq:ratiorho2}
\frac{\rho_{\nu+1,b}(x)}{\rho_{\nu,b}(x)}=x\frac{\rho_{\nu-1,b}(x)}{\rho_{\nu,b}(x)}+\frac{\nu}{b}
=x\left(\frac{\rho_{\nu-1,b}(x)}{\rho_{\nu,b}(x)}+\frac{\nu}{bx}\right)
=x\int_{-\infty}^{0}\frac{\ud \sigma_4(t)}{x-t},
\end{equation}
where
$$\ud \sigma_4 =-\left(\frac{1}{b \pi^2(J_{\nu}^2(2b\sqrt{-t})+Y_{\nu}^2(2b\sqrt{-t}))t}-\frac{\nu}{b}\delta_0\right) \ud t$$
is a measure on $(-\infty,0]$ with both a continuous part and a discrete part.

Combining \eqref{eq:rationomega1}--\eqref{eq:ratiorho2} and Theorem \ref{thm:CouVan}, we conclude that both $(\omega_{\mu,a},\omega_{\mu+1,a})$ and $(\rho_{\nu,b},\rho_{\nu+1,b})$ form AT systems on $(0,\infty)$. Theorem \ref{thm:nomality} then follows from this conclusion, as shown in \cite{FMM13}, where
the matrix of measures $\ud \textbf{S}:=\textbf{w}_2^T\textbf{w}_1\ud x$ is called an AT matrix measure. For the convenience of the readers, we include the proof here.

To show that the function $Q_{\textbf{n}_1,\textbf{n}_2}$ has $|\textbf{n}_2|$ sign changes on $(0,\infty)$, we first observe from \eqref{eq:orthoQ1} and \eqref{eq:orthoQ2} that
\begin{equation}\label{eq:orthoThmNor}
\int_0^\infty (p_1(x)\rho_{\nu,b}(x)+p_2(x)\rho_{\nu+1,b}(x))Q_{\textbf{n}_1,\textbf{n}_2}(x)\ud x=0
\end{equation}
for any polynomials $p_i$ with $\deg p_i\leq n_{2,i}-1$, $i=1,2$. Suppose that $Q_{\textbf{n}_1,\textbf{n}_2}$
changes sign at the $k$ points $\{x_1,\ldots,x_k\}\in(0,\infty)$. If $k<|\textbf{n}_{2}|$, it is always possible to find
a multi-index $\textbf{n}_3=(n_{3,1},n_{3,2})$ such that $|\textbf{n}_3|=k$, $n_{3,1}<n_{2,1}$ and $n_{3,2}\leq n_{2,2}$.
Consider the interpolation problem that the function
$$P(x)=q_1(x)\rho_{\nu,b}(x)+q_2(x)\rho_{\nu+1,b}(x), ~~\deg q_1=n_{3,1}, ~~
\deg q_2=n_{3,2}-1,$$
satisfies the interpolation conditions
$$P(x_i)=0, ~~i=1,\ldots, k; \qquad P(x_0)=1, ~~ \textrm{for some other point $x_0\in\mathbb(0,+\infty)$}.$$
Since $(\rho_{\nu,b},\rho_{\nu+1,b})$ is an AT system for the multi-index $\textbf{n}_3+\textbf{e}_1=(n_{3,1}+1,n_{3,2})$, this interpolation problem has a unique solution. Furthermore, $P$ has exactly $k$ sign changes at the same points as $Q_{\textbf{n}_1,\textbf{n}_2}$ and has no other sign changes on $(0,\infty)$. Thus, the function $P(x)Q_{\textbf{n}_1,\textbf{n}_2}(x)$ does not change sign on $(0,\infty)$ and
$$\int_0^\infty P(x)Q_{\textbf{n}_1,\textbf{n}_2}(x) \ud x \neq 0,$$ which contradicts \eqref{eq:orthoThmNor}. As a consequence, $k\geq |\textbf{n}_2|$. Note that $(\omega_{\mu,a},\omega_{\mu+1,a})$ also forms an AT system for $\textbf{n}_1$, it then follows from Definition \ref{def:AT} that $Q_{\textbf{n}_1,\textbf{n}_2}$ has at most $|\textbf{n}_1|-1=|\textbf{n}_2|$ zeros. We therefore conclude that $k=|\textbf{n}_2|$.

For the normality, we suppose that there exists $i\in\{1,2\}$ such that $\deg A_{\textbf{n}_1,\textbf{n}_2,i}<n_{1,i}-1$.
Then $Q_{\textbf{n}_1,\textbf{n}_2}$ would have at most $|\textbf{n}_1|-2=|\textbf{n}_2|-1$ zeros, since $(\omega_{\mu,a},\omega_{\mu+1,a})$ is an AT system for any multi-index $\textbf{n}_1-\textbf{e}_i$, which is again a contradiction.

This completes the proof of Theorem \ref{thm:nomality}
\qed

\paragraph{}
Due to the facts \eqref{eq:rationomega2} and \eqref{eq:ratiorho2}, the polynomials $ A_{\textbf{n}_1,\textbf{n}_2,1}$ and $ A_{\textbf{n}_1,\textbf{n}_2,2}$ can also be interpreted as mixed type multiple orthogonal polynomials with respect to a pair of Nikishin systems \cite{Sorokin94}. It has been shown in \cite{FL11a,FL11b} that Nikishin systems are perfect under various mild constraints on the measures, however, the results therein do not apply directly in the present case.

Finally, we point out that one could easily conclude results similar to Theorem \ref{thm:nomality} for the vector of polynomials $(B_{\textbf{n}_1,\textbf{n}_2,1}, B_{\textbf{n}_1,\textbf{n}_2,2})$ and the function $P_{\textbf{n}_1,\textbf{n}_2}$ by duality.

\section{Explicit formulas}\label{sec:expli formula}
From this section on, we will focus on the mixed type multiple orthogonal polynomials and their linear forms under the condition that each multi-index is on or close to the diagonal and investigate their properties. Recall the notations introduced at the end of Section \ref{sec:MMOP}, we start with an explicit formula for $Q_n$.

\subsection{Explicit formulas for $Q_n$, $A_{n,1}$ and $A_{n,2}$}
\begin{theorem}\label{thm:Qnexpli}
For $n\geq 1$ and $x>0$, we have
\begin{align}\label{eq:Qnexpli}
Q_n(x)
&=\frac{\det
\begin{pmatrix}
\Gamma(\mu+\nu+1) & \Gamma(\mu+\nu+2) & \cdots & \Gamma(\mu+\nu+1+n) \nonumber \\
\vdots &  \vdots & \vdots & \vdots \\
\Gamma(\mu+\nu+n) & \Gamma(\mu+\nu+n+1) & \cdots & \Gamma(\mu+\nu+2n) \nonumber \\
\omega_{\mu,a}(x) & \frac{b^2-a^2}{a}\omega_{\mu+1,a}(x) & \cdots &  \left(\frac{b^2-a^2}{a}\right)^n\omega_{\mu+n,a}(x)
\end{pmatrix} }{\prod_{k=0}^{n-1}k!\Gamma(\mu+\nu+1+k)} \nonumber \\
&=(-1)^n\sum_{j=0}^n \binom{n}{j}\frac{\Gamma(\mu+\nu+1+n)}{\Gamma(\mu+\nu+1+j)}\left(\frac{a^2-b^2}{a}\right)^j \omega_{\mu+j,a}(x).
\end{align}
Alternatively, the following integral representations for $Q_n$ holds:
\begin{align}\label{eq:intofQn}
Q_n(x)&=(-1)^n\frac{(\mu+\nu+1)_n}{a^\mu} \nonumber \\
&~~ \times \frac{1}{2\pi
i}\int_{c-i\infty}^{c+i\infty}{\; }_2 F_1 \left({ -n, t+\mu \atop \mu+\nu+1} \Big{|} 1-\frac{b^2}{a^2} \right)
\frac{\Gamma(t+\mu)}{\Gamma(1-t)} \sin\left(\pi \left(t+\mu+\frac{1}{2}\right)\right)(a^2 x)^{-t} \ud t
\nonumber \\
&=(ax)^\mu \Gamma(\mu+\nu+1+n)n!
\nonumber \\
&~~ \times\oint_{\Sigma}{\; }_0 F_1 \left({ - \atop t+\mu+1} \Big{|} a^2x \right)
\frac{\Gamma(t-n)[(b^2-a^2)x]^t}{\Gamma(t+1)\Gamma(t+\mu+\nu+1)\Gamma(t+\mu+1)} \ud t,
\end{align}
where $c>1$ and $\Sigma$ is a closed contour that encircles $0, 1, \ldots, n$ once in the positive direction.

Furthermore, one has
\begin{equation}\label{eq:Qnderivative}
\frac{\ud}{\ud x}Q_n^{\mu+1,\nu,a,b}(x)=aQ_n^{\mu,\nu,a,b}(x).
\end{equation}
\end{theorem}
\begin{proof}
By \eqref{def:An}--\eqref{def:QnNotation} and \eqref{eq:orthoQ1}--\eqref{eq:orthoQ2}, we need to verify the right hand side of \eqref{eq:Qnexpli} belongs to the linear span of $\{\omega_{\mu,a}(x),x\omega_{\mu,a}(x),\ldots, x^{\lfloor \frac{n}{2}\rfloor}\omega_{\mu,a}(x),\omega_{\mu+1,a}(x),x\omega_{\mu+1,a}(x),\ldots, \\
x^{\lfloor \frac{n-1}{2}\rfloor}\omega_{\mu,a}(x)\}$ and satisfies the orthogonality conditions
\begin{equation*}
\int_0^\infty Q_n(x)(p_1(x)\rho_{\nu,b}(x)+p_2(x)\rho_{\nu+1,b}(x))\ud x=0,
\end{equation*}
for any polynomial $p_{i}(x)$, $i=1,2$, with $\deg{p_1}\leq \lfloor \frac{n-1}{2}\rfloor$ and $\deg{p_2}\leq \lfloor \frac{n-2}{2}\rfloor$. The key observations here are, on account of Propositions \ref{prop:omegahigh} and \ref{prop:rhohigh},
\begin{align}
&\textrm{Span}\{\omega_{\mu,a}(x),\ldots, x^{\lfloor \frac{n}{2}\rfloor}\omega_{\mu,a}(x),\omega_{\mu+1,a}(x),\ldots,
x^{\lfloor \frac{n-1}{2}\rfloor}\omega_{\mu+1,a}(x)\} \nonumber \\
&\qquad \qquad \qquad \qquad \qquad \qquad \qquad \qquad=\textrm{Span}\{\omega_{\mu,a}(x),\ldots, \omega_{\mu+n,a}(x)\},\label{eq:spanequv1}\\
&\textrm{Span}\{\rho_{\nu,b}(x),\ldots, x^{\lfloor \frac{n-1}{2}\rfloor}\rho_{\nu,b}(x),\rho_{\nu+1,b}(x),\ldots,
x^{\lfloor \frac{n-2}{2}\rfloor}\rho_{\nu+1,b}(x)\} \nonumber \\
&\qquad \qquad \qquad \qquad \qquad \qquad\qquad \qquad=\textrm{Span}\{\rho_{\nu,b}(x),\ldots, \rho_{\nu+n-1,b}(x)\}.
\label{eq:spanequv2}
\end{align}
Hence, by normality, it is equivalent to check
\begin{equation}\label{eq:spanspace}
Q_n(x)\in\textrm{Span}\{\omega_{\mu,a}(x),\ldots, \omega_{\mu+n,a}(x)\},
\end{equation}
and satisfies the conditions
\begin{equation}\label{eq:orthocondQnequiv}
\int_0^\infty Q_n(x)\rho_{\nu+j,b}(x)\ud x=0, \qquad j=0,1,\ldots,n-1.
\end{equation}

The condition \eqref{eq:spanspace} follows directly from the determinantal representation in \eqref{eq:Qnexpli} and the fact that $b>a>0$. To show \eqref{eq:orthocondQnequiv}, we note the following definite integral involving both $I_\mu$ and $K_\nu$:
\begin{multline*}
\int_0^{\infty}x^{-\lambda}I_{\mu}(ax)K_{\nu}(bx)\ud x =\frac{a^\mu \Gamma(\frac{1-\lambda+\mu+\nu}{2})\Gamma(\frac{1-\lambda-\nu+\mu}{2})}{2^{\lambda+1}\Gamma(\mu+1)b^{-\lambda+\mu+1}}\\
\times {\; }_2 F_1 \left({\frac{1-\lambda+\mu+\nu}{2},\frac{1-\lambda-\nu+\mu}{2} \atop \mu+1} \Big{|} \frac{a^2}{b^2} \right)
\end{multline*}
for $\Re(\mu+1-\lambda\pm\nu)>0$ and $b>a$; see \cite[Equation 5, Page 676]{Tables}. This tells us that
\begin{multline}\label{eq:intomegarho}
\int_0^\infty \omega_{\mu+i,a}(x)\rho_{\nu+j,b}(x)\ud x=\int_0^\infty x^{\frac{\mu+\nu+i+j}{2}}I_{\mu+i}(2a\sqrt{x})
K_{\nu+j}(2b\sqrt{x})\ud x
\\ =\frac{a^{\mu+i}b^{\nu+j}\Gamma(\mu+\nu+1+i+j)}{2(b^2-a^2)^{\mu+\nu+1+i+j}},\qquad i,j\in\mathbb{Z}_+.
\end{multline}
Thus,
\begin{align}
&\int_0^\infty \det \begin{pmatrix}
\Gamma(\mu+\nu+1) & \Gamma(\mu+\nu+2) & \cdots & \Gamma(\mu+\nu+1+n) \nonumber \\
\vdots &  \vdots & \vdots & \vdots \\
\Gamma(\mu+\nu+n) & \Gamma(\mu+\nu+n+1) & \cdots & \Gamma(\mu+\nu+2n) \nonumber \\
\omega_{\mu,a}(x) & \frac{b^2-a^2}{a}\omega_{\mu+1,a}(x) & \cdots &  \left(\frac{b^2-a^2}{a}\right)^n\omega_{\mu+n,a}(x)
\end{pmatrix}\rho_{\nu+j,b}(x)\ud x \nonumber \\
&=\frac{a^\mu b^{\nu+j}}{2(b^2-a^2)^{\mu+\nu+1+j}} \nonumber \\
& ~~ \qquad \qquad \times \det \begin{pmatrix}
\Gamma(\mu+\nu+1) & \Gamma(\mu+\nu+2) & \cdots & \Gamma(\mu+\nu+1+n) \nonumber \\
\vdots &  \vdots & \vdots & \vdots \\
\Gamma(\mu+\nu+n) & \Gamma(\mu+\nu+n+1) & \cdots & \Gamma(\mu+\nu+2n) \nonumber \\
\Gamma(\mu+\nu+1+j) & \Gamma(\mu+\nu+2+j) & \cdots &  \Gamma(\mu+\nu+1+n+j)
\end{pmatrix} \nonumber \\
&=0 \nonumber
\end{align}
for $j=0,1,\ldots,n-1$, which gives us \eqref{eq:orthocondQnequiv}. By expanding the matrix along the last row and evaluating the associated minors we obtain the second equality in \eqref{eq:Qnexpli}.

To show the Mellin-Barnes integral representation of $Q_n$ in \eqref{eq:intofQn}, we start with the formula
$$I_{\mu}(x)=\left(\frac{x}{2}\right)^\mu \frac{1}{2\pi i}
\int_{c-i\infty}^{c+i\infty}\frac{\Gamma(t)
\sin\left(\pi \left(t+\frac{1}{2}\right)\right)}{\Gamma(1+\mu-t)}\left(\frac{x^2}{4}\right)^{-t}\ud t, \qquad c>0,$$
which can be easily verified with the help of the residue theorem and \eqref{def:I}. Combining this formula and
\eqref{def:omega}, we obtain by a change of variable that
\begin{equation}
\omega_{\mu+j,a}(x)=a^{-\mu-j}\frac{1}{2\pi i}\int_{c-i\infty}^{c+i\infty}
\frac{\Gamma(t+\mu+j)\sin\left(\pi \left(t+\mu+j+\frac{1}{2}\right)\right)}{\Gamma(1-t)}\left(a^2 x\right)^{-t}\ud t,
\end{equation}
for $j=0,1,\ldots,n$, and $c>-\mu-j>1$. This, together with \eqref{eq:Qnexpli}, implies that
\begin{align*}
&Q_n(x)
\\&=\frac{(-1)^n}{a^\mu}\frac{\Gamma(\mu+\nu+1+n)}{\Gamma(\mu+\nu+1)}\frac{1}{2\pi i} \\
&~~ \times \int_{c-i\infty}^{c+i\infty} \sum_{j=0}^n \frac{(-n)_j(t+\mu)_j\Gamma(t+\mu)}
{j!(\mu+\nu+1)_j\Gamma(1-t)}\left(\frac{a^2-b^2}{a^2}\right)^j
\sin\left(\pi \left(t+\mu+\frac{1}{2}\right)\right)
\left(a^2 x\right)^{-t}\ud t \\
&=(-1)^n\frac{(\mu+\nu+1)_n}{a^\mu} \nonumber \\
&~~ \times \frac{1}{2\pi i}\int_{c-i\infty}^{c+i\infty}{\; }_2 F_1 \left({ -n, t+\mu \atop \mu+\nu+1} \Big{|} 1-\frac{b^2}{a^2} \right)
\frac{\Gamma(t+\mu)}{\Gamma(1-t)} \sin\left(\pi \left(t+\mu+\frac{1}{2}\right)\right)(a^2 x)^{-t} \ud t.
\end{align*}
For the second equality in \eqref{eq:intofQn}, we note the following series expansion of $\omega_{\mu+j,a}$:
\begin{equation}\label{eq:expanomega}
\omega_{\mu+j,a}(x)=x^{\frac{\mu+j}{2}}I_{\mu+j}(2a\sqrt{x})
=(ax)^{\mu+j}\sum_{k=0}^\infty\frac{(a^2x)^k}{k!\Gamma(\mu+j+k+1)}, \end{equation}
which follows from $\eqref{def:I}$. Inserting this expansion into \eqref{eq:Qnexpli}, it is readily seen that
\begin{align}\label{eq:Qnintproof1}
Q_n(x)&=(-1)^n\sum_{j=0}^n \binom{n}{j}\frac{\Gamma(\mu+\nu+1+n)}{\Gamma(\mu+\nu+1+j)}\left(\frac{a^2-b^2}{a}\right)^j \sum_{k=0}^\infty\frac{(ax)^{\mu+j}(a^2x)^k}{k!\Gamma(\mu+j+k+1)} \nonumber
\\
&=(-1)^n(ax)^\mu\Gamma(\mu+\nu+1+n) \nonumber \\
&\qquad \qquad \times \sum_{k=0}^\infty\frac{(a^2x)^k}{k!}\sum_{j=0}^n\frac{(-n)_j}{j!}\frac{[(b^2-a^2)x]^j}
{\Gamma(\mu+j+k+1)\Gamma(\mu+\nu+1+j)}.
\end{align}
An easy calculation with the aid of the residue theorem shows that
\begin{multline}\sum_{j=0}^n\frac{(-n)_j}{j!}\frac{[(b^2-a^2)x]^j}
{\Gamma(\mu+j+k+1)\Gamma(\mu+\nu+1+j)}\\
=(-1)^nn!\oint_{\Sigma}\frac{\Gamma(t-n)[(b^2-a^2)x]^t}{\Gamma(t+1)\Gamma(t+\mu+k+1)\Gamma(t+\mu+\nu+1)}\ud t,
\end{multline}
where $\Sigma$ is a closed contour encircling $0, 1, \ldots, n$ once in the positive direction.
Substituting the above formula into \eqref{eq:Qnintproof1}, we obtain by interchanging the summation and integral that
\begin{align*}\label{eq:Qnintproof2}
Q_n(x)&=(ax)^\mu\Gamma(\mu+\nu+1+n)n! \nonumber \\
&\qquad  \times\oint_{\Sigma}\frac{\Gamma(t-n)[(b^2-a^2)x]^t}{\Gamma(t+1)\Gamma(t+\mu+\nu+1)\Gamma(t+\mu+1)} \sum_{k=0}^\infty\frac{(a^2x)^k}{k!(t+\mu+1)_k}\ud t \\
&=(ax)^\mu \Gamma(\mu+\nu+1+n)n! \\ &\qquad \times\oint_{\Sigma}\frac{\Gamma(t-n)((b^2-a^2)x)^t}{\Gamma(t+1)\Gamma(t+\mu+\nu+1)\Gamma(t+\mu+1)}{\; }_0 F_1 \left({ - \atop t+\mu+1} \Big{|} a^2x \right) \ud t,
\end{align*}
as required.

Finally, the differential property \eqref{eq:Qnderivative} follows directly from \eqref{eq:Qnexpli} and \eqref{eq:recuromega2}.

This completes the proof of Theorem \ref{thm:Qnexpli}.
\end{proof}

We next come to the explicit formulas for the mixed type multiple orthogonal polynomials $A_{n,1}$ and $A_{n,2}$.

\begin{theorem}\label{thm:expAn}
With the polynomials $A_{n,i}$, $i=1,2$ defined in \eqref{def:An}, we have, for $n\geq 1$,
\begin{equation}\label{eq:explAn1}
A_{n,1}(x)=(-1)^n\Gamma(\mu+\nu+1+n)\sum_{i=0}^{\lfloor\frac{n}{2}\rfloor}a_{i,n}x^i,
\end{equation}
where
$$a_{i,n}=\left\{
            \begin{array}{ll}
              \frac{1}{\Gamma(\mu+\nu+1)}, & \hbox{$i=0$,} \\
              a^{2i}\sum_{j=2i}^n\binom{n}{j}\binom{j-i-1}{i-1}\frac{(\mu+i+1)_{j-2i}}{\Gamma(\mu+\nu+1+j)}
\left(\frac{b^2-a^2}{a^2}\right)^j, & \hbox{$i\geq 1$,}
            \end{array}
          \right.
$$
and
\begin{equation}\label{eq:explAn2}
A_{n,2}(x)=(-1)^{n+1}\Gamma(\mu+\nu+1+n)a\sum_{i=0}^{\lfloor\frac{n-1}{2}\rfloor}\tilde a_{i,n}x^i,
\end{equation}
where
$$\tilde a_{i,n}=\sum_{j=2i+1}^n\binom{n}{j}\binom{j-i-1}{i}\frac{(\mu+i+1)_{j-2i-1}}{\Gamma(\mu+\nu+1+j)}
\left(\frac{b^2-a^2}{a^2}\right)^j,\qquad i\geq 0.
$$
\end{theorem}
\begin{proof}From \eqref{eq:Qnexpli} and Proposition \ref{prop:omegahigh}, it is readily seen that
\begin{multline*}
Q_n(x)=(-1)^n\sum_{j=0}^n \binom{n}{j}\frac{\Gamma(\mu+\nu+1+n)}{\Gamma(\mu+\nu+1+j)}\left(\frac{a^2-b^2}{a}\right)^j
\\
\times\left(a^{-j} r_{j,\mu}(a^2 x)\omega_{\mu,a}(x)+a^{1-j}s_{j,\mu}(a^2 x)\omega_{\mu+1,a}(x)\right).
\end{multline*}
This, together with \eqref{def:QnNotation}, implies that
\begin{align*}
A_{n,1}(x)&=(-1)^n\sum_{j=0}^n \binom{n}{j}\frac{\Gamma(\mu+\nu+1+n)}{\Gamma(\mu+\nu+1+j)}\left(\frac{a^2-b^2}{a^2}\right)^j
r_{j,\mu}(a^2 x), \\
A_{n,2}(x)&=(-1)^n a \sum_{j=0}^n \binom{n}{j}\frac{\Gamma(\mu+\nu+1+n)}{\Gamma(\mu+\nu+1+j)}\left(\frac{a^2-b^2}{a^2}\right)^j
s_{j,\mu}(a^2 x).
\end{align*}
The formulas \eqref{eq:explAn1} and \eqref{eq:explAn2} then follow from substituting \eqref{def:rmmu} and \eqref{def:smmu} into the above two equations and rearranging the expansions.

This completes the proof of Theorem \ref{thm:expAn}.
\end{proof}

\subsection{Explicit formulas for $P_n$, $B_{n,1}$ and $B_{n,2}$}
There are similar explicit formulas for the dual functions $P_n$, $B_{n,1}$ and $B_{n,2}$.
\begin{theorem}\label{thm:Pnexpli}
For $n\geq 1$ and $x>0$, we have
\begin{align}\label{eq:Pnexpli}
P_n(x)
&=c_n\frac{\det
\begin{pmatrix}
\Gamma(\mu+\nu+1) & \cdots & \Gamma(\mu+\nu+n) & \rho_{\nu,b}(x) \nonumber \\
\Gamma(\mu+\nu+2) &  \cdots &  \Gamma(\mu+\nu+n+1)& \frac{b^2-a^2}{b}\rho_{\nu+1,b}(x) \nonumber \\
\vdots &  \vdots & \vdots & \vdots \\
\Gamma(\mu+\nu+1+n) & \cdots &\Gamma(\mu+\nu+2n) & \left(\frac{b^2-a^2}{b}\right)^n\rho_{\nu+n,b}(x)
\end{pmatrix}}{\prod_{k=0}^{n-1}k!\Gamma(\mu+\nu+k+1)} \nonumber \\
&=(-1)^n\frac{2(b^2-a^2)^{\mu+\nu+1}}{a^\mu b^\nu n!}\sum_{j=0}^n \binom{n}{j}\left(\frac{a^2-b^2}{b}\right)^j\frac{\rho_{\nu+j,b}(x)}{\Gamma(\mu+\nu+1+j)} ,
\end{align}
where $$c_n=\frac{2(b^2-a^2)^{\mu+\nu+1}}{a^\mu b^\nu  \Gamma(\mu+\nu+1+n)n!}.$$
Alternatively, the following Mellin-Barnes integral representation for $P_n$ holds:
\begin{multline}\label{eq:intofPn}
P_n(x)=(-1)^n\frac{(b^2-a^2)^{\mu+\nu+1}}{a^\mu b^{2\nu}\Gamma(\mu+\nu+1)n!} \\ \times\frac{1}{2\pi i}\int_{c-i\infty}^{c+i\infty}{\; }_2 F_1 \left({ -n, t+\nu \atop \mu+\nu+1} \Big{|} 1-\frac{a^2}{b^2} \right)\Gamma(t)
\Gamma(t+\nu) (b^2 x)^{-t} \ud t
\end{multline}
with $c>0$ and $x>0$.

Furthermore, one has
\begin{equation}\label{eq:Pnderivative}
\frac{\ud}{\ud x}P_n^{\mu,\nu+1,a,b}(x)=-bP_n^{\mu,\nu,a,b}(x)
\end{equation}
and
\begin{equation}\label{eq:orthoPQ}
\int_0^\infty Q_n(x)P_m(x)\ud x=\delta_{n,m}, \qquad n,m\in\mathbb{N},
\end{equation}
where $Q_n$ is given \eqref{eq:Qnexpli}.
\end{theorem}
\begin{proof}
The proof of \eqref{eq:Pnexpli} is similar to that of \eqref{eq:Qnexpli}. By \eqref{eq:spanequv1} and \eqref{eq:spanequv2},
it suffices to verify
\begin{equation}
P_n(x)\in\textrm{Span}\{\rho_{\nu,b}(x),\ldots, \rho_{\nu+n,b}(x)\},
\end{equation}
and satisfies the conditions
\begin{equation}\label{eq:Pnortho}
 \int_0^\infty P_n(x)\omega_{\mu+j,a}(x)\ud x=0, \qquad j=0,1,\ldots,n-1.
\end{equation}
These requirements follow directly from the determinantal representation and \eqref{eq:intomegarho}. A further expansion of the matrix leads to the second equality in \eqref{eq:Pnexpli}.

To show \eqref{eq:intofPn}, we note that
\begin{equation}
\rho_{\nu+j,b}(x)=\frac{b^{-\nu-j}}{4 \pi i}\int_{c-i\infty}^{c+i\infty}\Gamma(t)\Gamma(t+\nu+j)
\left(b^2 x\right)^{-t}\ud t,
\quad c>0,
\end{equation}
see \eqref{eq:rhoMellin}. This, together with \eqref{eq:Pnexpli}, implies that
\begin{align*}
P_n(x)&=(-1)^n\frac{(b^2-a^2)^{\mu+\nu+1}}{a^\mu b^{2\nu} n!}\sum_{j=0}^n \binom{n}{j}\left(\frac{a^2-b^2}{b^2}\right)^j\frac{1}{\Gamma(\mu+\nu+1+j)} \\
&\qquad \qquad \qquad \qquad \times \frac{1}{2 \pi i}\int_{c-i\infty}^{c+i\infty}\Gamma(t)\Gamma(t+\nu+j)
\left(b^2 x\right)^{-t}\ud t \\
&=(-1)^n\frac{(b^2-a^2)^{\mu+\nu+1}}{a^\mu b^{2\nu}\Gamma(\mu+\nu+1)n!} \\
&\qquad \times \frac{1}{2\pi i}\int_{c-i\infty}^{c+i\infty}\sum_{j=0}^n \left(\frac{(-n)_j(t+\nu)_j}{j!(\mu+\nu+1)_j}
\left(\frac{b^2-a^2}{b^2}\right)^j\right)\Gamma(t) \Gamma(t+\nu) (b^2 x)^{-t} \ud t \\
&=(-1)^n\frac{(b^2-a^2)^{\mu+\nu+1}}{a^\mu b^{2\nu}\Gamma(\mu+\nu+1)n!} \\
&\qquad \times \frac{1}{2\pi i}\int_{c-i\infty}^{c+i\infty}{\; }_2 F_1 \left({ -n, t+\nu \atop \mu+\nu+1} \Big{|} 1-\frac{a^2}{b^2} \right)
\Gamma(t) \Gamma(t+\nu) (b^2 x)^{-t} \ud t, \end{align*}
as expected.

Finally, the differential property \eqref{eq:Pnderivative} follows directly from \eqref{eq:Pnexpli} and \eqref{eq:recurrho2}. For the biorthogonality properties \eqref{eq:orthoPQ}, by \eqref{eq:Pnortho}, it remains to check the case when $m=n$. Note that the coefficient of $\omega_{\mu+n,a}$ in \eqref{eq:Qnexpli} is $\left(\frac{b^2-a^2}{a}\right)^n$, it is readily seen from \eqref{eq:Pnexpli} and \eqref{eq:intomegarho} that
\begin{align*}
&\int_0^\infty Q_n(x)P_n(x)\ud x = \left(\frac{b^2-a^2}{a}\right)^n\int_0^\infty P_n(x)\omega_{\mu+n,a}(x)\ud x \\
&=\frac{c_n\left(\frac{b^2-a^2}{a}\right)^n}{\prod_{k=0}^{n-1}k!\Gamma(\mu+\nu+1+k)}  \\
&~~ \times
\det
\begin{pmatrix}
\Gamma(\mu+\nu+1) & \cdots & \Gamma(\mu+\nu+n) & \frac{a^{\mu+n}b^\nu}{2(b^2-a^2)^{\mu+\nu+n+1}}\Gamma(\mu+\nu+n+1) \nonumber \\
\Gamma(\mu+\nu+2) &  \cdots &  \Gamma(\mu+\nu+n+1)& \frac{a^{\mu+n}b^\nu}{2(b^2-a^2)^{\mu+\nu+n+1}}\Gamma(\mu+\nu+n+2) \nonumber \\
\vdots &  \vdots & \vdots & \vdots \\
\Gamma(\mu+\nu+1+n) & \cdots &\Gamma(\mu+\nu+2n) & \frac{a^{\mu+n}b^\nu}{2(b^2-a^2)^{\mu+\nu+n+1}}\Gamma(\mu+\nu+1+2n)
\end{pmatrix}\\
&~~=1.
\end{align*}

This completes the proof of Theorem \ref{thm:Pnexpli}.                                                                               \end{proof}

Combining Theorem \ref{thm:Pnexpli} and Proposition \ref{prop:rhohigh},
explicit formulas for the dual polynomials $B_{n,1}$ and $B_{n,2}$ are immediate, which are given in the following theorem
and the proof is omitted.
\begin{theorem}\label{thm:expBn}
With the polynomials $B_{n,i}$, $i=1,2$ defined in \eqref{def:Bn}, we have, for $n\geq 1$,
\begin{equation}\label{eq:explBn1}
B_{n,1}(x)=(-1)^n\frac{2(b^2-a^2)^{\mu+\nu+1}}{a^\mu b^\nu n!}\sum_{i=0}^{\lfloor\frac{n}{2}\rfloor}b_{i,n}x^i,
\end{equation}
where
$$b_{i,n}=\left\{
            \begin{array}{ll}
              \frac{1}{\Gamma(\mu+\nu+1)}, & \hbox{$i=0$,} \\
              b^{2i}\sum_{j=2i}^n\binom{n}{j}\binom{j-i-1}{i-1}\frac{(\nu+i+1)_{j-2i}}{\Gamma(\mu+\nu+1+j)}
\left(\frac{a^2-b^2}{b^2}\right)^j, & \hbox{$i\geq 1$,}
            \end{array}
          \right.
$$
and
\begin{equation}\label{eq:explBn2}
B_{n,2}(x)=(-1)^{n}\frac{2(b^2-a^2)^{\mu+\nu+1}}{a^\mu b^{\nu-1} n!}\sum_{i=0}^{\lfloor\frac{n-1}{2}\rfloor}\tilde b_{i,n}x^i,
\end{equation}
where
$$\tilde b_{i,n}=\sum_{j=2i+1}^n\binom{n}{j}\binom{j-i-1}{i}\frac{(\nu+i+1)_{j-2i-1}}{\Gamma(\mu+\nu+1+j)}
\left(\frac{a^2-b^2}{b^2}\right)^j,\qquad i\geq 0.
$$
\end{theorem}

\section{Limiting forms of $Q_n$ and $P_n$}\label{sec:limitingform}
Based on the explicit formulas for $Q_n$ and $P_n$ established in Section \ref{sec:expli formula}, we are able to derive various limiting forms of $Q_n$ and $P_n$.

\begin{theorem}\label{thm:Qlim}
We have
\begin{equation}\label{eq:Qnazero}
\lim_{a \to 0, \; b/\sqrt{a} \to k}Q_n\left(\frac{x}{a}\right)
=(-1)^n\frac{(\mu+\nu+1)_n}{\Gamma(\mu+1)} {\; }_1 F_2 \left({ -n \atop \mu+\nu+1, \mu+1} \Big{|} k^2x \right)x^\mu,
\end{equation}
uniformly for $x$ in any compact subset of $(0,+\infty)$, where the notation $\lim_{a \to 0, \; b/\sqrt{a} \to k}$
means that both $a,b \to 0$ with $b/\sqrt{a}\to k>0$.

Suppose that $c=b-a>0$ fixed, one has
\begin{equation}\label{eq:Qnainfty}
\lim_{a \to +\infty, \; b-a=c}\frac{2\sqrt{\pi a}}{e^{2a x}}Q_n(x^2)=x^{\mu-\frac{1}{2}}L_n^{(\mu+\nu)}(2cx),
\end{equation}
uniformly for $x$ in any compact subset of $(0,+\infty)$, where the notation $\lim_{a \to +\infty, \; b-a=c}$ means that
both $a,b \to +\infty$ with $b-a=c>0$ and
$$
L_n^{(\alpha)}(x)
=(-1)^n\sum_{j=0}^n \binom{n}{j}\frac{\Gamma(\alpha+1+n)}{\Gamma(\alpha+1+j)}\left(-x \right)^j,\qquad \alpha>-1,
$$
is the monic generalized Laguerre polynomial.
\end{theorem}
\begin{proof}
From \eqref{eq:expanomega}, it is easily seen that
$$\lim_{a\to 0}\omega_{\mu+j,a}\left(\frac{x}{a}\right)=\frac{x^{\mu+j}}{\Gamma(\mu+1+j)},
\qquad j=0,1,\ldots,n,
$$
uniformly valid for $x$ belonging to any compact subset of $(0,+\infty)$. This, together with \eqref{eq:Qnexpli} and the fact that
$$\lim_{a \to 0, \; b/\sqrt{a} \to k}\frac{a^2-b^2}{a}=-k^2,$$
implies
\begin{align*}
&\lim_{a \to 0, \; b/\sqrt{a} \to k}Q_n\left(\frac{x}{a}\right) \nonumber
\\&=(-1)^n\Gamma(\mu+\nu+1+n)
\sum_{j=0}^n\frac{(-n)_j}{j!}\frac{x^{\mu+j}}{\Gamma(\mu+\nu+1+j)\Gamma(\mu+1+j)}k^{2j}  \nonumber \\
&=(-1)^n\frac{(\mu+\nu+1)_n}{\Gamma(\mu+1)}
{\; }_1 F_2 \left({ -n \atop \mu+\nu+1, \mu+1} \Big{|} k^2x \right)x^\mu,
\end{align*}
which is \eqref{eq:Qnazero}.

To show \eqref{eq:Qnainfty}, we note that, on account of \eqref{eq:asyI2} and \eqref{def:omega},
$$\lim_{a\to +\infty}\frac{2\sqrt{\pi a}}{e^{2ax}}\omega_{\mu+j,a}(x^2)=x^{\mu+j-\frac{1}{2}},
\qquad j=0,1,\ldots,n,$$
uniformly valid for $x$ belonging to any compact subset of $(0,+\infty)$, and
$$\lim_{a \to +\infty, \; b-a=c}\frac{a^2-b^2}{a}=-2c.$$
Substituting the above two equations into \eqref{eq:Qnexpli} gives us \eqref{eq:Qnainfty}.

This completes the proof of Theorem \ref{thm:Qlim}.
\end{proof}

The results for $P_n$ are stated in the following theorem.
\begin{theorem}\label{thm:Plim}
Suppose that $c=b-a>0$ fixed, we have
\begin{equation}\label{eq:Pnainfty}
\lim_{b \to +\infty, \; b-a=c}\sqrt{\frac{4b}{\pi}}e^{2 b x}P_n(x^2)=
\frac{2(b^2-a^2)^{\mu+\nu+1}}{a^\mu b^\nu  \Gamma(\mu+\nu+1+n)n!}x^{\nu-\frac{1}{2}}L_n^{(\mu+\nu)}(2cx),
\end{equation}
uniformly for $x$ in any compact subset of $(0,+\infty)$.

Furthermore, one has
\begin{equation}\label{eq:Pnxsmall}
\lim_{x \to 0}P_n(x)=(-1)^n\frac{(b^2-a^2)^{\mu+\nu+1}\Gamma(\nu)}{a^\mu b^{2\nu}\Gamma(\mu+\nu+1)n!}
{\; }_2 F_1 \left({ -n, \nu \atop \mu+\nu+1} \Big{|} 1-\frac{a^2}{b^2} \right)
\end{equation}
and the following Mehler-Heine asymptotics
\begin{multline}\label{eq:PnMHasy}
\lim_{n\to \infty} (-1)^nn!(n+1)^{\nu}P_n\left(\frac{x}{(n+1)(b^2-a^2)}\right)
=\frac{(b^2-a^2)^{\mu+1}}{a^\mu}\mathop{{G^{{2,0}}_{{0,3}}}\/}\nolimits\!\left({- \atop 0, \nu, -\mu} \Big{|} x\right)
\\=\frac{(b^2-a^2)^{\mu+1}}{a^\mu}\frac{1}{2\pi i}\int_{c-i\infty}^{c+i\infty}
\frac{\Gamma(t)\Gamma(t+\nu)}{\Gamma(1+\mu-t)}x^{-t} \ud t, \quad c>0,
\end{multline}
where $\mathop{{G^{{m,0}}_{{0,q}}}\/}$ stands for the Meijer G-function; cf. \cite[Section 16.17]{DLMF}.
\end{theorem}
\begin{proof}
By \eqref{eq:asyK2} and \eqref{def:rho}, it is readily seen that
\begin{equation}\label{eq:limrho}
\lim_{b\to +\infty}\sqrt{\frac{4b}{\pi}}e^{2 b x}\rho_{\nu+j,b}(x^2)=x^{\nu+j-\frac{1}{2}},
\qquad j=0,1,\ldots,n,
\end{equation}
uniformly valid for $x$ belonging to any compact subset of $(0,+\infty)$. A combination of \eqref{eq:limrho},
\eqref{eq:Pnexpli} and the fact that
$$\lim_{b \to +\infty, \; b-a=c}\frac{a^2-b^2}{b}=-2c$$
gives us \eqref{eq:Pnainfty}.

On account of \eqref{eq:asyrho1}, the limit \eqref{eq:Pnxsmall} is immediate from \eqref{eq:Pnexpli}. To
show \eqref{eq:PnMHasy}, we make use of the following asymptotics of Gauss hypergeometric function for large parameter
(see \cite[Equation 2.6]{Temme})
\begin{equation}
{\; }_2 F_1 \left({ -n, \alpha+\beta+1 \atop 2\alpha+1} \Big{|} 1- z \right)
\sim \frac{\Gamma(2\alpha+1)}{\Gamma(\alpha-\beta)}\left[(n+1)(1-z)\right]^{-\alpha-\beta-1},
\end{equation}
for $\Re (\alpha+\beta)>-1$ and $|z|<1$ fixed. Thus, in view of the Mellin-Barnes integral representation \eqref{eq:intofPn} for $P_n$, we have
\begin{align}
&\lim_{n\to \infty} (-1)^nn!(n+1)^{\nu}P_n\left(\frac{x}{(n+1)(b^2-a^2)}\right) \nonumber \\
&=\frac{(b^2-a^2)^{\mu+\nu+1}}{a^\mu b^{2\nu} \Gamma(\mu+\nu+1)} \nonumber \\
&\quad~~ \times \frac{1}{2 \pi i}\int_{c-i\infty}^{c+i\infty}\frac{\Gamma(\mu+\nu+1)\Gamma(t)\Gamma(t+\nu)}{\Gamma(1+\mu-t)}\left(\frac{b^2 x}{b^2-a^2}\right)^{-t}\left(1-\frac{a^2}{b^2}\right)^{-t-\nu} \ud t \nonumber \\
&= \frac{(b^2-a^2)^{\mu+1}}{a^\mu} \frac{1}{2 \pi i}\int_{c-i\infty}^{c+i\infty}\frac{\Gamma(t)\Gamma(t+\nu)}{\Gamma(1+\mu-t)}x^{-t} \ud t, \nonumber
\end{align}
which is \eqref{eq:PnMHasy}. The interchange of limit and integral is justified by the fact that the gamma function is
exponentially small along the contour of integration and the dominated convergence theorem.

This completes the proof of Theorem \ref{thm:Plim}.
\end{proof}

\section{Recurrence relations}\label{sec:recurrence rel}
By general theory of mixed type multiple orthogonal polynomials (cf. \cite{AFM11}), it follows that both $Q_n$ and $P_n$ satisfy five-term recurrence relations. It is the aim of this section to calculate these recurrence coefficients explicitly.
We stress out that both $Q_n$ and $P_n$ are not polynomials.

\begin{proposition}\label{prop:recurrence}
For $n\geq 0$, the function $Q_n$ satisfies the following recurrence relation:
\begin{equation} \label{eq:Qnrecurrence}
    x Q_n(x) = a_{2,n}Q_{n+2}(x) + a_{1,n}Q_{n+1}(x)+a_{0,n}Q_n(x)+a_{-1,n}Q_{n-1}(x)+a_{-2,n}Q_{n-2}(x)
\end{equation}
with initial conditions $Q_{-2}(x)=Q_{-1}(x)=0$ and $Q_{0}(x)=\omega_{\mu,a}(x)$, where
\begin{align*}
a_{2,n}&=\left(\frac{a}{b^2-a^2}\right)^2, \\
a_{1,n}&=2(\mu+\nu+2n+2)\left(\frac{a}{b^2-a^2}\right)^2+\frac{\mu+n+1}{b^2-a^2}, \\
a_{0,n}&=(6n^2+6(\mu+\nu+1)n+(\mu+\nu+1)(\mu+\nu+2))\left(\frac{a}{b^2-a^2}\right)^2 \nonumber
\\& \qquad +\frac{3n^2+(4\mu+2\nu+3)n+(\mu+1)(\mu+\nu+1)}{b^2-a^2},
\\
a_{-1,n}&=n(\mu+\nu+n)\left(2(\mu+\nu+2)\left(\frac{b}{b^2-a^2}\right)^2-\frac{\nu+n}{b^2-a^2}\right), \\
a_{-2,n}&=(n-1)n(\mu+\nu+n-1)(\mu+\nu+n)\left(\frac{b}{b^2-a^2}\right)^2.
\end{align*}
Similarly, the dual function $P_n$ satisfies recurrence relation
\begin{equation} \label{eq:Pnrecurrence}
    x P_n(x) = b_{2,n}P_{n+2}(x) + b_{1,n}P_{n+1}(x)+b_{0,n}P_n(x)+b_{-1,n}P_{n-1}(x)+b_{-2,n}P_{n-2}(x)
\end{equation}
with initial conditions $P_{-2}(x)=P_{-1}(x)=0$ and
$P_{0}(x)=\frac{2(b^2-a^2)^{\mu+\nu+1}}{a^\mu b^\nu\Gamma(\mu+\nu+1)}\rho_{\nu,b}(x)$, where
\begin{equation}\label{eq:aibi}
b_{i,n}=a_{-i,n+i}, \qquad i\in\{-2,-1,0,1,2\}.
\end{equation}
\end{proposition}
\begin{proof}
By \eqref{eq:orthocondQnequiv}, \eqref{eq:Pnortho} and \eqref{eq:orthoPQ}, it is readily seen that
\begin{equation}
\int_0^\infty Q_n(x)P_m(x)\ud x=\delta_{n,m}, \qquad n,m\in\mathbb{Z}_+.
\end{equation}
Hence,
\begin{align*}
    a_{i,n} = \int_0^{\infty} x Q_n(x)P_{n+i}(x)\ud x, \quad
  b_{i,n}  = \int_0^{\infty} x P_n(x) Q_{n+i}(x) \ud x, \quad i\in\{-2,-1,0,1,2\},
    \end{align*}
which gives us \eqref{eq:aibi}.

The explicit formulas of recurrence coefficients follow directly from the explicit formula for $Q_n$ and
the three-term recurrence relation for $\omega_{\mu,a}$. Indeed, by setting
$$c_{j,n}= (-1)^n \binom{n}{j}\frac{\Gamma(\mu+\nu+1+n)}{\Gamma(\mu+\nu+1+j)}\left(\frac{a^2-b^2}{a}\right)^j,
~~j=0,1,\ldots,n,$$
we obtain from \eqref{eq:Qnexpli}, \eqref{eq:recuromega1} and \eqref{eq:Qnrecurrence} that
\begin{align*}
    x Q_n(x) & = \sum_{j=0}^n c_{j,n}x\omega_{\mu+j,a}(x)
    =\sum_{j=0}^n c_{j,n}\left(\omega_{\mu+j+2,a}(x)+\frac{\mu+j+1}{a}\omega_{\mu+j+1,a}(x)\right)
\\
&=c_{0,n}\frac{\mu+1}{a}\omega_{\mu+1,a}(x)+\sum_{j=2}^{n+1}\left(c_{j-2,n}+c_{j-1,n}\frac{\mu+j}{a}\right)
\omega_{\mu+j,a}(x)+c_{n,n}\omega_{\mu+n+2,a}(x)
\\
&=a_{2,n}\sum_{j=0}^{n+2} c_{j,n+2}\omega_{\mu+j,a}(x)+a_{1,n}\sum_{j=0}^{n+1} c_{j,n+1}\omega_{\mu+j,a}(x)
+a_{0,n}\sum_{j=0}^{n} c_{j,n}\omega_{\mu+j,a}(x)
\\
&\qquad + a_{-1,n}\sum_{j=0}^{n-1} c_{j,n-1}\omega_{\mu+j,a}(x)+a_{-2,n}\sum_{j=0}^{n-2} c_{j,n-2}\omega_{\mu+j,a}(x).
\end{align*}
Comparing the coefficients of $\omega_{\mu+n+i,a}$, $i\in\{2,1,0,-1,-2\}$ in the last equality then gives us $a_{i,n}$
recursively.

This completes the proof of Proposition \ref{prop:recurrence}.
\end{proof}

By \eqref{def:QnNotation} and \eqref{def:PnNotation}, we have that the polynomials $A_{n,i}$ and $B_{n,i}$, $i=1,2$ also satisfy the recurrence relations \eqref{eq:Qnrecurrence} and \eqref{eq:Pnrecurrence}, respectively.

\section{Products of two coupled random matrices}\label{sec:prod}
In this section, we explain in detail how the mixed type multiple orthogonal polynomials associated with the modified Bessel functions are related to products of two coupled random matrices in a general setting introduced by Liu in \cite{Liu16}.

We start with a coupled two-matrix model defined by the probability distribution
\begin{equation}\label{def:coupledmatrix}
\frac{1}{\widehat Z_n} \exp \left(-\beta \textrm{Tr}(X_1X_1^*+X_2^*X_2)+\textrm{Tr}(\Omega X_1 X_2 +(\Omega X_1 X_2)^*)\right)\ud X_1\ud X_2,
\end{equation}
over pairs of rectangular complex matrices $(X_1,X_2)$, each of size $L\times M$ and $M\times n$ respectively, where $\ud X_1$ and $\ud X_2$ are the flat complex Lebesgue measures on the entries of $X_1$ and $X_2$, and $\widehat Z_n$ is a normalization constant. Here $\beta>0$, and $\Omega$ is a fixed $n\times L$ complex matrix such that $\Omega\Omega^* < \beta^2$ which plays the role of coupling between $X_1$ and $X_2$.

If $L=n$ and $\Omega$ is a scalar matrix, the model \eqref{def:coupledmatrix} can be interpreted as the chiral two-matrix model \cite{ADOS07,OSborn04}, which was introduced in the context of quantum chromodynamics. In this case, an alternative formulation of the model is given in \eqref{eq:X1andX2}. Indeed, the pair $(X_1,X_2)$ therein is then distributed according to \eqref{def:coupledmatrix} with $L=n$, $\beta=\frac{1+\tau}{2\tau}$
and $\Omega=\frac{1-\tau}{2\tau}I_n$, where $I_n$ stands for the $n\times n$ identity matrix; see \cite{AS15,Liu16}.

Our interest lies in the singular values of the product matrix
$$\widehat Y=X_1X_2,$$
where the pair $(X_1,X_2)$ has the probability distribution \eqref{def:coupledmatrix}.
It comes out that the squared singular values of $\widehat Y$  are distributed according to a determinantal point process \cite{AS15,Liu16} over the positive real axis. The determinantal point process is a biorthogonal ensemble \cite{Bor} with joint probability density function (see \cite[Proposition 1.1]{Liu16})
\begin{equation}
\frac{1}{Z_n}\det \left[I_{\kappa}(2\alpha_i\sqrt{x_j})\right]_{i,j=1}^n
\det\left[x_j^{\frac{\nu+i-1}{2}}K_{\nu-\kappa+i-1}(2\beta\sqrt{x_j})\right]_{i,j=1}^n,
\end{equation}
with $I_\mu$ and $K_\nu$ being the modified Bessel functions of the first kind and the second kind, respectively, where $$\kappa:=L-n, \qquad \nu:=M-n,$$
$\alpha_i$, $i=1,\ldots,n$ are the singular values of coupling matrix $\Omega$, and $Z_n$ is a normalization constant explicitly known. Here, it is also assumed that $L,M\geq n$.

We now focus on the confluent case that all the singular values of $\Omega$ are the same, that is, $\alpha_i \to \alpha>0$.
The linear space spanned by the functions $x \mapsto I_{\kappa}(2\alpha_i\sqrt{x})$, $i=1,\ldots,n$, tends to
the linear space spanned by
\begin{equation}\label{eq:linear_span}
x \mapsto \frac{\partial^{j-1}}{\partial y^{j-1}}I_{\kappa}(2y \sqrt{x})|_{y=\alpha}, \quad j=1,\ldots,n.
\end{equation}
Using the recurrence relations \eqref{eq:recurI1} and \eqref{eq:recurI2} satisfied by the modified Bessel
functions of the first kind, it is easily seen that the resulting space is spanned by the functions
$x \mapsto x^{\frac{i-1}{2}}I_{\kappa+i-1}(2\alpha\sqrt{x})$, $i=1,\ldots,n$. Thus, a further simple algebraic calculation
implies that the joint probability density function for the squared singular values of $\widehat Y$ is given by
\begin{equation}\label{eq:jpdfcon}
\frac{1}{Z_n}\det \left[x_j^{\frac{\kappa+i-1}{2}}I_{\kappa+i-1}(2\alpha\sqrt{x_j})\right]_{i,j=1}^n
\det\left[x_j^{\frac{\nu-\kappa+i-1}{2}}K_{\nu-\kappa+i-1}(2\beta\sqrt{x_j})\right]_{i,j=1}^n,
\end{equation}
under the condition that the coupling matrix $\Omega$ has a single singular value $\alpha$. For the case $\kappa=0$, the result was first obtained by Akemann and Strahov \cite{AS15}.

From general properties of biorthogonal ensembles \cite{Bor}, it is known that the correlation kernel of determinantal point process \eqref{eq:jpdfcon}
is given by \eqref{def:Kn},
where for each $k = 0, 1, \ldots$, $\mathcal{Q}_k$ belongs to the linear span of $x^{\frac{\kappa+i}{2}}I_{\kappa+i}(2\alpha \sqrt{x})$, $i=0,\ldots,k$, while $\mathcal{P}_k$ belongs to
the linear span of $x^{\frac{\nu-\kappa+i}{2}}K_{\nu-\kappa+i}(2\beta \sqrt{x})$, $i=0,\ldots,k$ in such a way that the biorthogonality conditions \eqref{eq:biortho} hold.
By Theorems \ref{thm:Qnexpli} and \ref{thm:Pnexpli}, we actually have
$$ \mathcal{Q}_k(x)=Q_{k}^{\kappa,\nu-\kappa,\alpha,\beta}(x),\qquad
\mathcal{P}_k(x)=P_{k}^{\kappa,\nu-\kappa,\alpha,\beta}(x),\qquad k=0,1,\ldots,$$
recall the notations in \eqref{eq:notations}. For the special case that $\kappa=0$, the explicit formulas and recurrence
relations for $\mathcal{Q}_k$ and $\mathcal{P}_k$ can also be found in \cite{AS15} and coincide with ours up to some common constants, but without noting the multiple orthogonality.

Due to the connection to multiple orthogonal polynomials of mixed type, the point process \eqref{eq:jpdfcon} is a multiple orthogonal polynomial ensemble \cite{Kui10a,Kui10b}. This in particular implies the following RH characterization (\cite{DK,DKV08,Delvaux10}) of the correlation kernel \eqref{def:Kn}.
\begin{rhp} \label{rhp:Y}
We look for a $4\times 4$ matrix-valued function
$Y : \mathbb C \setminus \mathbb [0,+\infty) \to \mathbb C^{4 \times 4}$
satisfying
\begin{enumerate}
\item[\rm (1)] $Y$ is defined and analytic in $ \mathbb{C} \setminus [0,+\infty)$.
\item[\rm (2)] $Y$ has limiting values $Y_{\pm}$ on $(0,\infty)$,
where $Y_+$ ($Y_-$) denotes the limiting value from the upper
(lower) half-plane, and
\begin{equation}\label{defjumpmatrix0}
Y_{+}(x) = Y_{-}(x)
\begin{pmatrix} I_2 & W(x)\\
0 & I_2
\end{pmatrix}, \qquad  x \in \mathbb (0,+\infty),
\end{equation}
where  $W(x)$ is
the rank-one matrix (outer product of two vectors)
\begin{align*}
W(x) &=
\begin{pmatrix}
\omega_{\kappa,\alpha}(x) \\
\omega_{\kappa+1,\alpha}(x)
\end{pmatrix}
\begin{pmatrix}
\rho_{\nu-\kappa,\beta}(x) &
\rho_{\nu-\kappa+1,\beta}(x)
\end{pmatrix}\\
&=\begin{pmatrix}
\omega_{\kappa,\alpha}(x)\rho_{\nu-\kappa,\beta}(x) & \omega_{\kappa,\alpha}(x)\rho_{\nu-\kappa+1,\beta}(x) \\
\omega_{\kappa+1,\alpha}(x)\rho_{\nu-\kappa,\beta}(x) & \omega_{\kappa+1,\alpha}(x)\rho_{\nu-\kappa+1,\beta}(x)
\end{pmatrix},
\end{align*}
where $\omega_{\mu,a}$ and $\rho_{\nu,b}$ are given in \eqref{def:omega} and \eqref{def:rho}, respectively.

\item[\rm (3)] As $z\to\infty$, we have that
\begin{equation}\label{asymptoticconditionY0}
    Y(z) = \left(I_4+\frac{Y_1}{z}+O\left(\frac{1}{z^2}\right)\right)
    \diag(z^{n_1},z^{n_2},z^{-n_1},z^{-n_2}).
\end{equation}
with $n_1=\lfloor \frac{n-1}{2} \rfloor+1$ and $n_2=\lfloor \frac{n-2}{2} \rfloor+1$.
\end{enumerate}
\end{rhp}
Combined with appropriate local behavior near the origin that depends on the parameters
$\kappa,\nu$, it is shown in \cite{DK} that the RH problem for $Y$
has a unique solution and the correlation kernel \eqref{def:Kn} admits the following representation in terms of the solution of the RH problem:
\begin{multline}\label{kernel representation}
K_{n}(x,y)
\\=\frac{1}{2\pi i(x-y)}\begin{pmatrix}0 &0 & \rho_{\nu-\kappa,\beta}(y)&
\rho_{\nu-\kappa+1,\beta}(y)\end{pmatrix} Y_{+}^{-1}(y)Y_{+}(x)
\begin{pmatrix}
\omega_{\kappa,\alpha}(x) \\ \omega_{\kappa+1,\alpha}(x) \\ 0 \\ 0
\end{pmatrix}.
\end{multline}
The representation \eqref{kernel representation} is based on the Christoffel-Darboux
formula for multiple orthogonal polynomials of mixed type; see also \cite{AFM11}.

The solution $Y(z)$ to the RH problem actually admits the following partition:
\begin{equation}
Y(z)=\begin{pmatrix}
Y_{1,1}(z) & Y_{1,2}(z) \\
Y_{2,1}(z) & Y_{2,2}(z)
\end{pmatrix},
\end{equation}
where each block $Y_{i,j}$ is of size $2 \times 2$. The blocks $Y_{1,1}$ and $Y_{2,1}$ are built in terms of mixed type
multiple orthogonal polynomials, while the blocks $Y_{1,2}$ and $Y_{2,2}$ contain certain Cauchy transforms thereof;
see \cite{DK} for the precise description. A further connection between mixed type multiple orthogonal polynomials and the products of coupled random matrices is that
\begin{equation}
\mathbb{E}\left[\prod_{k=1}^n(z-x_k)\right]=\det Y_{1,1}(z),
\end{equation}
where the expectation $\mathbb{E}$ is taken with respect to \eqref{eq:jpdfcon}, and the polynomial is called average characteristic polynomial; see \cite[Theorem 1.2]{Delvaux10}.

Based on an integral representation of the correlation kernel $K_n$, the hard edge scaling limit as well as its transition  have been established in \cite{AS15b,AS15,Liu16}. The hard edge scaling limit belongs to the Meijer G-kernels found in the products of independent random matrices \cite{Kuijlaars-Zhang14}. The limiting mean distribution and local universality of the squared singular values of $\widehat Y$, however, remain open. The interpretation of \eqref{eq:jpdfcon} as a multiple orthogonal polynomial ensemble then provides an alternative way to tackle this problem by performing Deift/Zhou steepest descent analysis \cite{Deift99book} for the associated RH problem $Y$. The study of this aspect will be the topics of future research.

\section*{Acknowledgment}
The author would like to thank Dang-Zheng Liu comments on the manuscript. The work was partially supported by The Program for Professor of Special Appointment (Eastern Scholar) at Shanghai Institutions of Higher Learning (No. SHH1411007), by National Natural Science Foundation of China (No. 11501120) and by Grant EZH1411513 from Fudan University.



\begin{thebibliography}{10}

\bibitem{ADOS07}
G. Akemann, P.~H.~Damgaard, J.~C. Osborn and K. Splittorff, A new chiral two-matrix
theory for Dirac spectra with imaginary chemical potential, Nucl. Phys. B 766 (2007),
34--76.

\bibitem{AS15b}
G. Akemann and E. Strahov, Hard edge limit of the product of two strongly coupled random
matrices, preprint arXiv:1511.09410.

\bibitem{AS15}
G. Akemann and E. Strahov, Dropping the independence: singular values for products of two
coupled random matrices, Comm. Math. Phys. 345 (2016), 101--140.

\bibitem{AFM11}
C. \'{A}lvarez-Fern\'{a}ndez, U. Fidalgo and M. Ma\~{n}as, Multiple orthogonal polynomials of mixed type: Gauss-Borel factorization and the multi-component 2D Toda hierarchy, Adv. Math. 227 (2011), 1451--1525.

\bibitem{Apt98} A.~I. Aptekarev,
Multiple orthogonal polynomials,
J. Comput.\ Appl.\ Math. 99 (1998), 423--447.

\bibitem{ABVAN03}
A.~I. Aptekarev, A. Branquinho and W. Van Assche,
Multiple orthogonal polynomials for classical weights, Trans. Amer. Math. Soc. 355 (2003), 3887--3914.

\bibitem{CD00}
Y. Ben Cheikh and K. Douak,
On two-orthogonal polynomials realted to the Bateman's $J_n^{u,v}-$fuction,
Meth. Appl. Anal.  7  (2000), 641--662.

\bibitem{Bor}
A. Borodin,
Biorthogonal ensembles,
Nucl. Phys. B 536 (1999), 704--732.

\bibitem{BLOP92}
J. Bustamante and G.~L\'{o}pez Lagomasino, Hermite-Pad\'{e} approximation to a Nikishin type system of analytic
functions, Mat. Sb. 183 (1992), 117--138 (in Russian); English translation in: Russian Acad. Sci. Sb.
Math. 77 (1994), 367--384.

\bibitem{Chihara}
T.~S. Chihara,
An Introduction to Orthogonal Polynomials, New York: Gordon and Breach, 1978.

\bibitem{CCVA08}
E. Coussement, J. Coussement and W. Van Assche,
Asymptotic zero distribution for a class of multiple orthogonal polynomials,
Trans. Amer. Math. Soc. 360 (2008), 5571--5588.

\bibitem{CouVan03}
E. Coussement and W. Van Assche,
Multiple orthogonal polynomials associated with the modified Bessel functions of the first kind, Constr. Approx. 19 (2003),
237--263.

\bibitem{CouVan03b}
E. Coussement and W. Van Assche,
Asymptotics of multiple orthogonal polynomials associated with the modified Bessel functions of the first kind, J. Comput.\ Appl.\ Math. 153 (2003), 141--149.

\bibitem{CVA2001}
 E. Coussement and W. Van Assche,
 Some properties of multiple orthogonal polynomials associated with Macdonald functions,
 J. Comput. Appl. Math.  133 (2001), 253--261.

\bibitem{DK}
E.~Daems and A.~B.~J.~Kuijlaars, Multiple orthogonal
polynomials of mixed type and non-intersecting {B}rownian motions,
J. Approx. Theory 146 (2007), 91--114.

\bibitem{DKV08}
E.~Daems, and A.~B.~J. Kuijlaars and W. Veys, Asymptotics of non-intersecting Brownian motions and
a 4 $\times$ 4 Riemann-Hilbert problem, J. Approx. Theory 153 (2008), 225--256.

\bibitem{Deift99book}
P.~Deift, Orthogonal polynomials and random matrices: a
{R}iemann-{H}ilbert approach, vol.~3 of Courant Lecture Notes in
Mathematics, New York University Courant Institute of Mathematical
Sciences, New York, 1999.

\bibitem{Delvaux10}
S. Delvaux,
Average characteristic polynomials for multiple orthogonal polynomial ensembles,
J. Approx. Theory 162 (2010), 1033--1067.

\bibitem{DKRZ12}
S. Delvaux, A. B. J. Kuijlaars, P. Rom\'{a}n and L. Zhang,
 Non-intersecting squared Bessel paths with one positive starting
 and ending point, J. Anal. Math. 118 (2012), 105--159.

\bibitem{DS94}
K.~Driver and H.~Stahl,
\newblock Normality in Nikishin systems, Indag. Math. (N.S.) 5 (1994), 161--187.


\bibitem{FL11a}
U.~Fidalgo and G.~L\'{o}pez Lagomasino,
\newblock Nikishin systems are perfect,
Constr. Approx. 34 (2011), 297--356.

\bibitem{FL11b}
U.~Fidalgo and G.~L\'{o}pez Lagomasino,
\newblock Nikishin systems are perfect. The case of unbounded and touching supports,
J. Approx. Theory 163 (2011), 779--811.

\bibitem{FMM13}
U.~Fidalgo, S.~Medina Peralta and J.~M\'{i}nguez Ceniceros,
\newblock Mixed type multiple orthogonal polynomials: Perfectness
and interlacing properties of zeros, Linear Algebra Appl. 438 (2013), 1229--1239.

\bibitem{Tables}
I.~S.~Gradshteyn and I.~M.~Ryzhik, Table of Integrals,
Series, and Products (6th ed.), San Diego, CA: Academic
Press Inc., translated from the Russian, 2000.

\bibitem{Grosswald}
E. Grosswald, The student t distribution of any degrees of freedom is infinitely divisible, Z. Wahrscheinlichkeitstheorie und Verw. Gebiete 36 (1976), 103--109.


\bibitem{Ismail}
M.~E.~H. Ismail,
\newblock Classical and Quantum Orthogonal Polynomials in One
Variable, Encyclopedia of Mathematics and its Applications
98, Cambridge University Press, 2005.

\bibitem{Ismail77}
M.~E.~H. Ismail, Bessel functions and the infinite divisibility of the student t distribution,
Ann. Prob. 5 (1977), 582--585.

\bibitem{IsmailKel79}
M.~E.~H. Ismail and D.~H. Kelker, Special functions, Stieltjes transforms and infinite divisibility,
SIAM J. Math. Anal. 10 (1979), 884--901.

\bibitem{Kui10a}
A.~B.~J. Kuijlaars, Multiple orthogonal polynomial ensembles, in:
Recent Trends in Orthogonal Polynomials and Approximation Theory (J.
Arves\'u, F. Marcell\'an and A. Mart\'inez-Finkelshtein eds.),
Contemp. Math. 507 (2010), 155--176.

\bibitem{Kui10b}
A.~B.~J. Kuijlaars,
Multiple orthogonal polynomials in random matrix theory,
in: Proceedings of the International Congress of Mathematicians, Volume III
(R. Bhatia, ed.) Hyderabad, India, 2010, pp. 1417--1432.

\bibitem{KMW09}
 A.~B.~J. Kuijlaars, A. Mart\'inez Finkelstein and F. Wielonsky,
 Non-intersecting squared Bessel paths and multiple orthogonal polynomials for modified Bessel weights,
 Comm. Math. Phys.  286 (2009), 217--275.

\bibitem{KuiRom10}
A.~B.~J. Kuijlaars and P. Rom\'{a}n, Recurrence relations and vector equilibrium problems arising from a model of
non-intersecting squared Bessel paths, J. Approx. Theory 162 (2010), 2048--2077.

\bibitem{Kuijlaars-Zhang14}
A.~B.~J. Kuijlaars and L.~Zhang,
\newblock Singular values of products of {G}inibre random matrices, multiple
  orthogonal polynomials and hard edge scaling limits,
Comm. Math. Phys. 332 (2014), 759--781.

\bibitem{Liu16}
D.-Z. Liu, Singular values for products of two coupled random matrices:
hard edge phase transition, preprint arXiv:1602.00634.

\bibitem{Niki82}
E.~M.~Nikishin, On simultaneous Pad\'{e} approximations, Mat. Sb. 113 (155) (1980), 499--519
(Russian); Math. USSR Sb. 41 (1982), 409--425.

\bibitem{NikSor}
E.~M.~Nikishin and V.~N.~Sorokin, Rational Approximations and Orthogonality, in: Translations
of Mathematical Monographs 92, Amer.\ Math.\ Soc.\,
Providence RI, 1991.

\bibitem{DLMF}
F.~W.~J. Olver, D.~W. Lozier, R.~F. Boisvert and C.~W. Clark, editors,
NIST Handbook of Mathematical Functions, Cambridge
University Press, Cambridge 2010. Print companion to [DLMF].

\bibitem{OSborn04}
J.~C. Osborn, Universal results from an alternate random matrix model for QCD with a
baryon chemical potential, Phys. Rev. Lett. 93 (2004), 222001--222004.

\bibitem{Sorokin94}
V.~N.~Sorokin, Hermite-Pad\'{e} approximants of polylogarithms, Izv. Vyssh. Uchebn. Zaved. Mat. 38
(1994), 49--59.

\bibitem{Temme}
N.~M. Temme,
\newblock Uniform asymptotic expansion for a class of polynomials biorthogonal on the
unit circle, Constr. Approx. 2 (1986), 369--376.




\bibitem{WVA06}
W. Van Assche, Pad\'{e} and Hermite-Pad\'{e} approximation
and orthogonality, Surv. Approx. Theory  2 (2006), 61--91.

\bibitem{VAY00}
W. Van Assche and S.~B. Yakubovich,
Multiple orthogonal polynomials associated with Macdonald functions,
Integral Transform. Spec. Funct. 9 (2000), 229--244.

\bibitem{Watson66}
G.~N. Watson, A Treatise on the Theory of Bessel Functions, Cambridge University
Press, Cambridge, 1966.

\bibitem{Zhang13}
L.~Zhang, \newblock A note on the limiting mean distribution of singular values for
  products of two {W}ishart random matrices,
\newblock  J. Math. Phys. 54 (2013), 083303 8 pp.

\bibitem{ZP}
L. Zhang and P. Rom\'{a}n, The asymptotic zero distribution of multiple orthogonal polynomials
associated with Macdonald functions, J. Approx. Theory 163 (2011), 143--162.


\end{thebibliography}
\end{document}